\documentclass[12pt,a4paper]{amsart}
\usepackage[utf8]{inputenc} %

\title{Maker-Breaker games on infinite graphs with precolored edges}
\author[N.~Bowler]{Nathan Bowler}
\author[F.~Gut]{Florian Gut}
\author[H.~Ortm{\"u}ller]{Henri Ortm{\"u}ller}
\email{nathan.bowler@uni-hamburg.de, henri.ortmueller@unifr.ch}

\date{\today}

\usepackage{graphicx} %

\usepackage[left=25mm,right=25mm,top=30mm,bottom=30mm,includeheadfoot]{geometry} %
\usepackage{amsmath,amsthm,amssymb} %
\usepackage{thmtools} %
\usepackage{mathtools} %
\usepackage{dsfont} %
\usepackage[colorlinks, linkcolor = blue]{hyperref} %
\usepackage[nameinlink]{cleveref}
\usepackage{microtype} %
\usepackage{esvect} %
\usepackage{caption} %
\usepackage{subcaption} %
\usepackage{todonotes} %
\usepackage[backend=biber]{biblatex} %
\usepackage{tikz, tikz-network} %
\usepackage[many]{tcolorbox} %
\usepackage{enumitem} %
\setlist[enumerate]{label=(\roman*)}
\usepackage{float} %
\allowdisplaybreaks %
\usepackage[makeroom]{cancel}

\newcommand{\N}{\mathbb{N}} %

\newcommand{\Mod}[1]{\ (\mathrm{mod}\ #1)}

\newcommand{\SalephNull}{S_{\aleph_0}}
\newcommand{\KalephNull}{K_{\aleph_0}}
\newcommand{\alephNull}{\aleph_0}

\newcommand{\MBcolorGraphs}[1]{\text{MB}_{\text{col}}(G_1, \dots, G_{#1})}
\newcommand{\MBcolor}[2]{\text{MB}_{\text{col}}(#1, #2)}

\newcommand{\MBpatternGraphs}[1]{\text{MB}_{\text{pat}}(G_1, \dots, G_{#1})}

\newcommand{\MBpartpattern}[4]{\text{MB}^{#2}_{\text{par}}(#1, #3, #4)}
\newcommand{\MBpartpatternNumber}[3]{f_{\text{par}}^{#2}(#1, #3)}

\newcommand{\MBauxM}[2]{\text{MB}^{(#2)}_{\text{aux}}(#1)}
\newcommand{\MBauxB}[2]{\overline{\text{MB}}^{(#2)}_{\text{aux}}(#1)}
\newcommand{\MBauxVertexM}[3]{\text{MB}^{(#2)}_{\text{aux}}(#1, #3)}
\newcommand{\MBauxVertexB}[3]{\overline{\text{MB}}^{(#2)}_{\text{aux}}(#1, #3)}

\newcommand{\MBAscInfStars}[3]{\text{MB}_{\ref{thm:(part)pattStar}}(#1, #2, #3)}
\newcommand{\MBfin}[3]{\text{MB}_{\text{fin}}(#1, #2, #3)}

\newtheorem{theorem}{Theorem}[section] 
\newtheorem{lemma}[theorem]{Lemma}

\newtheorem{corollary}[theorem]{Corollary}
\theoremstyle{definition}
\newtheorem{definition}[theorem]{Definition}

\newtheorem{remark}[theorem]{Remark}
\newtheorem{problem}[theorem]{Problem}

\newenvironment{caseproof}[1]
{%
  \begin{proof}[Proof of Case #1:]
  
}
{%
  \end{proof}
}

\begin{document}

\begin{abstract}
    Suppose we are given graphs $B$ and $G$. 
    In the classical \emph{Maker-Breaker game} $\text{MB}(B,G)$ two players, Maker and Breaker, alternately claim edges of $B$ and it is Maker's goal to claim a copy of $G$ in $B$, while it is Breaker's goal to prevent that.
    In this paper, $B$ is the countably infinite complete graph $K_{\aleph_0}$ and we are given finitely many infinite subgraphs $G_1, \dots, G_k \subseteq B$. 
    In the color preserving game, it will be Maker's goal to claim a $K_{\aleph_0} \subseteq B$, which contains infinitely many edges of each $G_i$.
    We present sufficient winning conditions for both Maker and Breaker, if $k > 1$ and a full characterization of the game, if $k =1$. 
    This partly answers a question of Bowler, Emde and Gut \cite{Bowler_2023}.
    In the (partially) pattern preserving game, it is Maker's goal to claim a copy $K$ of $K_{\aleph_0}$, such that $G_i \cap K$ is isomorphic to (a subgraph of) $G_i$ for all $i \in [k]$. In those games, we investigate some patterns for which Maker has a winning strategy.
\end{abstract}

\maketitle

\section{Introduction}

In recent years, increasing attention has been drawn toward infinite structure preserving Maker-Breaker games \cite{Bowler_2023, Bowler_2024}.
In this paper, all Maker-Breaker games we study are played on the countably infinite board $B = \KalephNull$ unless specified otherwise. Moreover, Maker and Breaker will alternately claim edges of $B$ and Maker begins. It will always be Maker's goal to claim a copy of $\KalephNull$ in $B$ with some additional structure depending on the game and Breaker's goal to prevent Maker from doing so.

Since it has been shown by Bowler, Emde and Gut \cite{Bowler_2023} that Maker has a winning strategy in $\text{MB}(\KalephNull, \KalephNull)$, they considered in a next step the countably infinite board $B$ together with a precoloring of its vertices.
For the Maker-Breaker game in which Maker's goal is to claim a $\KalephNull$ containing every color infinitely often, they answered most of the open questions in \cite{Bowler_2023}. If one considers precolored edges instead, the situation is less clear.
In the game with precolored edges, whether Maker has a winning strategy, depends on the particular arrangement of the colors.
If we have a 2-coloring of $E(B)$ using black and white for example and the black edges form an infinite star, Maker has a winning strategy, whereas Breaker has a winning strategy if the black edges form an infinite matching.
In this paper, we investigate further, for which colorings Maker, and for which colorings Breaker has a winning strategy.
This (partially) answers a question of Bowler, Emde and Gut \cite[Question 6.1]{Bowler_2023}.

To formalize the game above, suppose we are given a coloring of the board $c \colon E(B) \to [k]$ for some $k \in \N$.
In the game $\MBcolor{k}{c}$, we say that Maker wins, if she can claim a copy $K \subseteq B$ of $\KalephNull$ that contains each color infinitely often, i.e.~ $|c^{-1}(\{i\}) \cap K| = \alephNull$ for all $i \in [k]$. For $k = 2$, we can give the following complete characterization of which player has a winning strategy in $\MBcolor{k}{c}$.
\begin{theorem}
    \label{thm:2ColoringCol}
    Given a coloring $c \colon E(B) \to [2]$. If the graph induced by $c^{-1}(\{1\})$ or $c^{-1}(\{2\})$ has finite maximum degree, Breaker has a winning strategy in $\MBcolor{2}{c}$. Otherwise, Maker has a winning strategy in $\MBcolor{2}{c}$.
\end{theorem}

For three or more colors the situation is less clear. Observe that in the game $\MBcolor{k}{c}$, the graphs induced by $c^{-1}(\{i\})$ have to partition the edges of $\KalephNull$. We consider the more general game, we call the \emph{color preserving game} $\MBcolorGraphs{k}$, where instead of an edge coloring, we are given graphs $G_1, \dots, G_k \subseteq B$ (which neither have to be disjoint nor covering $E(B)$). In $\MBcolorGraphs{k}$, it is Maker's goal to claim a copy $K \subseteq B$ of $\KalephNull$, such that $|K \cap G_i| = \alephNull$ for all $i \in [k]$.
For $k \geq 3$, we will later provide conditions for which Maker (respectively Breaker) can win $\MBcolorGraphs{k}$ based on the structure of $G_1, \dots, G_k$.

Observe that by Ramsey's theorem \cite{RamseysThm}, one color in the 2-coloring in \Cref{thm:2ColoringCol} has to induce an infinite complete graph. This essentially reduces $\MBcolor{2}{c}$ to a game of the form $\text{MB}_{\text{col}}(G)$. If $G$ does not have bounded maximal degree, there either exists $v \in V(G)$ with $d_G(v) = \alephNull$ or a sequence of vertices $\{v_i\}_{i \in \N}$, such that $d_G(v_i) \geq i$ for all $i \in \N$. In both cases, we will not just prove that Maker can preserve an infinite number of edges in her $\KalephNull$, but something even stronger: Maker can (partially) preserve the structure of $G$, i.e.~a vertex of infinite degree or a sequence of vertices with increasing degrees, respectively.

Since the proofs concerning the color preserving game rely on which kind of structure Maker can preserve in her $\KalephNull$ and this is an interesting question in its own right, we phrase it explicitly:
For which graphs $G' \subseteq B$ can Maker preserve an entire copy of $G'$ in the $\KalephNull$ she claims?
Observe that Breaker has a trivial winning strategy if $E(G')$ is finite but of size at least two, by claiming an edge of $G'$ in his first move.
In this paper, we consider the following generalized version of the game above.

Suppose, we are given graphs $G_1, \dots, G_k \subseteq B$. We say that Maker has a winning strategy in the \emph{pattern preserving game} $\MBpatternGraphs{k}$, if she can claim a copy $K \subseteq B$ of $\KalephNull$, such that there exists a bijection $f \colon V(B) \to V(K)$ with
\begin{align*}
    \{x, y\} \in E(G_i) \Leftrightarrow \{f(x), f(y)\} \in E(G_i)
\end{align*}
for all $x,y \in V(B)$ and $i \in [k]$.

Let $K_j$ and $S_j$ be the clique and the star consisting of $j$ vertices, respectively. We define $G_{\text{cli}}$ and $G_{\text{star}}$ as the graphs consisting of pairwise vertex disjoint cliques $\{K_j\}_{j \in \N}$ and stars $\{S_j\}_{j \in \N}$, respectively. From results we present later, we obtain the following corollary.
\begin{corollary}
    \label{cor:patternCliqueStar}
    For $k \in \N$, let $G_1, \dots, G_k \subseteq B$ be pairwise vertex disjoint copies of $G_{\text{star}}$ or $G_{\text{cli}}$. Then, Maker has a winning strategy in $\MBpatternGraphs{k}$.
\end{corollary}
Let $\SalephNull$ be the star with infinitely many leaves. By using methods from Bowler, Emde and Gut \cite{Bowler_2023}, one can show that Maker has a winning strategy in $\text{MB}_{\text{pat}}(\SalephNull)$. What about $\text{MB}_{\text{pat}}(\SalephNull \sqcup \SalephNull)$? Let $a$ and $b$ be the vertices of infinite degree of $\SalephNull \sqcup \SalephNull$. Then, Breaker has a simple winning strategy: For each vertex $v \in B \setminus \{a,b\}$, Breaker claims the edge $vb$ or $va$ immediately after Maker has claimed the edge $va$ or $vb$, respectively.
Since Maker does not have a winning strategy in the pattern preserving game, we investigate how many disjoint copies of $\SalephNull$ are needed in $B$, such that Maker can claim a copy $K \subseteq B$ of $\KalephNull$ containing two (or more) copies of $\SalephNull$. From \Cref{thm:(part)pattStar}, it will follow that 3000 copies of $\SalephNull$ suffice in order for Maker to claim a $\KalephNull$ that contains two copies of $\SalephNull$.\footnote{A more careful analysis using the ideas of this paper shows that the actual minimal number of stars that we need is six.}
Understanding how many copies of $\SalephNull$ are needed for Maker to win will be essential for our results on $\MBcolorGraphs{k}$. To make such games precise, we use the following definition.

Suppose, we are given a graph $H$ and integers $k, \ell, C \in \N$. Let $G_1, \dots, G_k \subseteq B$ be pairwise vertex disjoint graphs with each $G_i$ consisting of $C$ pairwise vertex disjoint copies of $H$.
In the \emph{partially pattern preserving game} $\MBpartpattern{k}{H}{C}{\ell}$, it is Maker's goal to claim a copy $K \subseteq B$ of $\KalephNull$, such that $G_i \cap K$ contains at least $\ell$ copies of $H$.
We denote with $\MBpartpatternNumber{k}{H}{\ell}$ the minimal integer $a$, such that Maker has a winning strategy in $\MBpartpattern{k}{H}{a}{\ell}$, if it exists. Otherwise, we set $\MBpartpatternNumber{k}{H}{\ell} = \infty$.

Above, we have stated that $2< \MBpartpatternNumber{1}{\SalephNull}{2} < 3000$. With the methods developed in this paper, we can generalize this as follows.
\begin{theorem}
    \label{thm:partpattGame}
    For $k, \ell \in \N$, we have $2^{(1 - o(1)) \frac{k \ell}{2}} < \MBpartpatternNumber{k}{\SalephNull}{\ell} \leq 2^{2^{(1+o(1))k \ell}}$.
\end{theorem}

We can even combine the pattern preserving game and the partially pattern preserving game.
Even though the following result seems very concrete at first, it will turn out to be central for most of Maker's winning strategies presented in this paper.
Let $G_{\text{inf}}^C$, be the graph consisting of $C$ pairwise vertex disjoint copies of $\SalephNull$.
\begin{theorem}
    \label{thm:(part)pattStar}
    Given $s,t \in \N_0$ and pairwise vertex disjoint graphs $G_1, \dots, G_{s+t} \subseteq B$, such that for $C = \max\{\lceil s^2 \cdot 2^{s-3} \rceil, 2^{3s+1} \cdot t\}$ and $C' \geq 3C (s+1)^{Cs + 1}$, we have for each $i \in [s]$ that $G_i$ is a copy of $G_{\text{inf}}^{C'}$ and each of $G_{s+1}, \dots, G_{s+t}$ is a copy of $G_{\text{star}}$. Maker has a strategy to claim a copy $K$ of $\KalephNull$, such that $G_i \cap K$ induces a copy of $\SalephNull$ for each $i \in [s]$ and $G_j \cap K$ induces a copy of $G_{\text{star}}$ for each $j \in \{s+1, \dots, s+t\}$.
\end{theorem}
In \Cref{ch:FinInfEquiv}, we are going to define a finite game for which we show that it is equivalent to the game presented in \Cref{thm:(part)pattStar}. The complexity of this finite characterization motivates why we should not expect a simple characterization of which player wins in $\MBcolorGraphs{k}$ for large $k$.

In order to state our main result regarding $\MBcolorGraphs{k}$, we need to distinguish a few types of graph.
Observe that for every graph $G$ with infinitely many edges exactly one of the following holds.
\begin{enumerate}
    \item Either there exists a sequence of distinct vertices $(v_j)_{j \in \N} \subseteq V(G)$, such that $d(v_j) \geq j$,
    \item or for some $C \in \N_0$, there exists a vertex set $S$ of size $C$, such that $\Delta(G - S) \in \N$ and $d_G(v) = \alephNull$ for all $v \in S$.
\end{enumerate}
If $G$ is of type (i), we call $G$ \emph{ascending} and if $G$ is of type (ii), we say $G$ \emph{is of order $C$}. We refer to an infinite graph $G$ of order 0 as \emph{bounded}. For a family of infinite graphs $G_1, \dots, G_k \subseteq B$ and $A \subseteq [k]$, let $m(A)$ be the size of a minimal vertex set $\{v_1, \dots, v_{m(A)}\}$, such that for each $j \in A$, there exists $v \in \{v_1, \dots, v_{m(A)}\}$ with $d_{G_j}(v) = \alephNull$. If there exists no such set, we let $m(A) = \infty$. With this preliminary work, we can state the main theorem.

\begin{theorem}
    \label{thm:kColoring}
    Given $s,t \in \N_0$ and $G_1, \dots, G_{s+t} \subseteq B$, such that $G_1, \dots, G_s$ are of order $C_1, \dots, C_s$ and $G_{s+1}, \dots, G_{s+t}$ are ascending, respectively. If
    \begin{align*}
        C_1, \dots, C_s \geq \max\{   2^{t \cdot 2^{(1+o_s(1))3s}}, 2^{2^{(1+o_s(1))s}}  \}
    \end{align*}
    Maker has a winning strategy in $\MBcolorGraphs{s+t}$. For the Euler number $e$, if there exists a non-empty set $A \subseteq [s]$, such that for $m(A)$ as above
    \begin{align*}
        \sum_{j \in A} C_j \leq  \frac{m(A)}{e} \cdot 2^{\frac{m(A)}{2}-1},
    \end{align*}
    then Breaker has a winning strategy in $\MBcolorGraphs{s+t}$. Moreover, Breaker's strategy ensures that $G' = \bigcup_{j \in A} G_j \cap K'$ is a graph of order less than $m(A)$ for each infinite clique $K' \subseteq B$ for which Maker has claimed $E(K')$.
\end{theorem}

\section{Definitions and Notation}

All (hyper)graphs considered in this paper are simple. 
We refer to Maker as \textit{she} and Breaker as \textit{he} and say that a vertex $x$ is \emph{fresh}, if no player has claimed an edge incident to $x$ yet.
In order to be concise, we will not give a formal definition of \emph{strategy} here. Instead, we refer the reader to \cite[Appendix C]{CombGamesBeck}. 
We define $[k] \coloneqq \{1, \dots, k\}$ and occasionally use that $k \coloneqq \{0, \dots, k-1\}$ for some $k \in \N$. In this paper a coloring will just be a function $c$. Whenever $c$ has additional properties, this will be stated explicitly.
For a finite family of graphs $\mathcal{G}$, we sometimes write $V(\mathcal{G})$ in order to denote $\bigcup_{G \in \mathcal{G}} V(G)$.
We define $E(A,B) \coloneqq \{ab \ | \ a \in A, b \in B \}$ for $A,B \subseteq V(G)$. For $x \in V(G)$ and $H$ a subgraph of $G$, we sometimes say that Maker connects $x$ to $H$. By that we mean that Maker claims all edges of $E(\{x\}, H)$. Similarly for Breaker.

\begin{definition}
    \label{def:linegraph}
    For a graph $G$, we define the \emph{line graph} $L(G)$ of $G$ by setting
    \begin{enumerate}
        \item $V(L(G)) = E(G)$,
        \item $E(L(G)) = \{e e'\ | \ e,e' \in E(G), \ e \cap e' \neq \emptyset \}$.
    \end{enumerate}
\end{definition}
For a graph $G$ and two vertices $u,v \in V(G)$, let $dist_G(u,v)$ be the length of a shortest path between $u$ and $v$ in $G$.
\begin{definition}
    \label{def:distGraph}
     For $k \in \N$, we define $G^k$ by setting 
     \begin{enumerate}
        \item $V(G^k) = V(G)$,
        \item $E(G^k) = \{uv \ | \ u,v \in V(G), \ 1 \leq dist_G(u,v) \leq k\}$.
     \end{enumerate}
\end{definition}

\begin{definition}
    \label{def:KCs}
    For some $C,s \in \N$, let $K_C^s$ be the complete $s$-partite graph with vertex set $\{r_1, \dots, r_{Cs}\}$, such that every partition class is of size $C$. For a set of edges $A \subseteq E(K_C^s)$, we define an $s$-uniform hypergraph $K_C^{(s)}(A)$ by putting $e \in K_C^{(s)}(A)$, if $e^{(2)} \subseteq A$ and set $K_C^{(s)} \coloneqq K_C^{(s)}(E(K_C^s))$.
\end{definition}
We define a \emph{ray} $R$ as an infinite graph with $V(R) \coloneqq \{r_i\}_{i \in \N}$ and $E(R) \coloneqq \{r_i, r_{i+1}\}_{i \in \N}$ with all $r_i$ distinct.
Let $T$ be a tree rooted at $r$ and let $x \in V(T)$. We denote $P_x$ as the unique rooted path from $x$ to $r$ in $T$. For $x \neq r$, we call the unique vertex in $N(x) \cap P_x$ the \emph{predecessor} of $x$ and denote it with $p(x)$.
\begin{definition}
    \label{def:wrg:orderontree}
    Given a tree $T$ rooted at $r$. Define a partial order $(V(T), \leq)$ by letting
    \begin{align*}
        x \leq y  :\Longleftrightarrow x \in P_y
    \end{align*}
    for $x,y \in T$. We will call this ordering the \textit{tree order} of $T$.
\end{definition}
Moreover, we define ${\lfloor x \rfloor}_T \coloneqq \{y \in V(T) \ | \ x \leq y\}$ for $x \in T$ and ${\lceil x \rceil}_T$ similarly. We will often omit the $T$, if it is clear from the context.
For a vertex $x$ of a rooted tree $T$, we call elements of $\lfloor x \rfloor \cap N(x)$ the $\mathit{children}$ of $x$ and define $ch \colon V(T) \to \N$ as the function, which returns the number of children of a vertex.
We call a tree $T$ \emph{$k$-ary}, if each vertex has exactly $k$ children or is a leaf and we will often refer to a $2$-ary tree as \emph{binary tree}. For a tree $T$ rooted at $r$ and $k \in \N$, we say that $T' \subseteq T$ is a \emph{maximal $k$-ary subtree} of $T$, if $r \in V(T')$ and $|V(T')|$ is maximal, such that $T'$ is $k$-ary.
\begin{definition}
    For $k,\ell \in \N$, we define the rooted tree $T(k,\ell)$ to be the unique $k$-ary tree, such that each rooted path contains $\ell$ edges.
\end{definition}
Let $T_k^{(\ell+1)}$ be the $(\ell+1)$-uniform hypergraph with vertex set $V(T(k,\ell))$ and the edge set given by the maximal rooted paths in $T(k,\ell)$.

\begin{definition}
    Given a rooted tree $T$, we define the \emph{comparability graph} $C_T$ of $T$ as the graph with the following two properties:
    \begin{enumerate}
        \item  $V(C_T) \coloneqq V(T)$;
        \item  $E(C_T) \coloneqq \{ \{x,y\} \in V(C_T)^{(2)} \ | \ y \in V(P_x) \setminus \{x\} \} $;
    \end{enumerate}
\end{definition}

\begin{definition}
    For a rooted tree $T$, let $U \colon V(T) \to \N$ be the function, which assigns for a vertex $x$ the number of vertices in its up-closure. Formally, $U(x) = |{\lfloor x \rfloor}_T|$.
\end{definition}
For a rooted tree $T$ and $x \in T$, let $h_T(x) \coloneqq |E(P_x)|$ be the \emph{height} of $x$. We sometimes omit the $T$, if it is clear from context. 
Moreover, we define the height of a tree as $h(T) \coloneqq \sup_{x \in V(T)}\{h_T(x)\}$ and for $n \in \N$, we set 
\begin{align*}
    U^*_n(x) \coloneqq \{v \in {\lfloor x \rfloor}_T \ | \ h(v) - h(x) \leq n\}.
\end{align*}

\begin{definition}
    \label{def:attach}
    For some $m \in \N$, given trees $T^*, T_1, \dots, T_m$ rooted at $r^*, r_1, \dots, r_m$, respectively and $v \in T^*$. Whenever we say, that â€œwe attach $T_1, \dots, T_m$ to $v$ (at their roots)", we mean that we will define a tree $T$ rooted at $r^*$, such that
    \begin{enumerate}
        \item  $V(T) \coloneqq V(T^*) \cup V(T_1) \cup \ldots \cup V(T_m)$,
        \item  $E(T) \coloneqq  E(T^*) \cup E(T_1) \cup \ldots \cup E(T_m) \cup \{v r_1, \dots, v r_m\}$.
    \end{enumerate}
    
\end{definition}

\section{Preliminaries}
In the proof of our main results, we will essentially connect vertices to subtrees of trees and the following lemma will provide a flexible tool for that.
\begin{lemma}
    \label{lem:connVertexToHalfTrees}
    For $\ell, m \in \N$ and a family of pairwise disjoint $2m$-ary trees $\mathcal{T} = \{T_1, \dots, T_{2 \ell}\}$ in $B$ rooted at $r_1, \dots, r_{2 \ell}$, respectively, suppose there is a vertex $x \in V(B) \setminus V(\mathcal{T})$, such that no edge of $E(\{x\}, V(\mathcal{T}))$ has been chosen by Maker or Breaker.
    If Breaker begins in the Maker-Breaker game on $B$, Maker has a strategy to claim all edges of $E(\{x\}, V(\mathcal{T'}))$, for $\mathcal{T'}$ being a family of $\ell$ maximal $m$-ary subtrees $T'_{j_1}, \dots, T'_{j_\ell}$ of $T_{j_1}, \dots, T_{j_\ell}$, respectively, such that $j_i \in \{2i-1, 2i\}$ for each $i \in [\ell]$. Moreover, Maker can achieve this in $|V(\mathcal{T'})|$ moves.
\end{lemma}
\begin{proof}
    We define a pairing of $\mathcal{T}$ as follows. For each $i \in [\ell]$, we pair $r_{2i-1}$ with $r_{2i}$. As long as there exists a vertex with unpaired children $w_1, \dots, w_{2m}$, we pair $w_{2j-1}$ with $w_{2j}$ for each $j \in [m]$. This yields a pairing $\mathcal{P}$ of $V(\mathcal{T})$, such that each vertex is contained in exactly one pair.

    Suppose that for $n \in \N_0$ after her $n^\text{th}$ move, Maker has connected $x$ to (possibly empty) rooted trees $T^{(n)}_{j_1}, \dots, T^{(n)}_{j_\ell}$, such that after her last move
    \begin{enumerate}
        \item Maker has connected $x$ to at most one vertex from each pair of $\mathcal{P}$,
        \item $\sum_{i \in [\ell]} |V(T^{(n)}_{j_i})| = n$,
        \item for each edge $xy$ Breaker has claimed, there is $z \in \lceil y \rceil$, such that Maker has claimed the edge from x to the partner of $z$ in $\mathcal{P}$.
    \end{enumerate}
    As a first case, suppose that Breaker now claims an edge $xz$ in his $(n+1)^\text{st}$ move. Let $z_1 \in \lceil z \rceil$ be the vertex of minimal height for which Maker has not claimed $xz_1$ and suppose further that Maker has also not claimed $xz_2$, where $z_2$ is the partner of $z_1$ in $\mathcal{P}$. Breaker cannot have claimed $xz_2$, since in that case by (iii) there would exist $(z_3, z_4) \in \mathcal{P}$ with $z_3 \in \lceil z_2 \rceil \setminus \{z_2\}$, such that Maker has claimed $xz_4$. Hence, Maker would have claimed both $xz_3$ and $xz_4$, contradicting (i). Therefore, Maker can claim $xz_2$ in her $(n+1)^\text{st}$ move.
    \begin{figure}[H]
        \centering
        \includegraphics[width=0.37\textwidth]{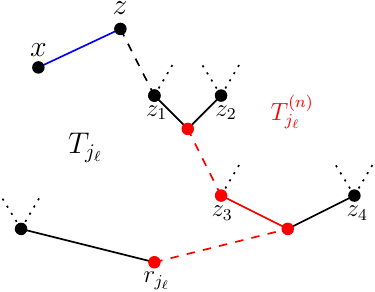}
        \caption{For $m = 1$, if Breaker claims $xz$, we have to verify in the first case that Breaker has not claimed $xz_2$ yet.}
    \end{figure}

    If Breaker claims any edge of $B$ not satisfying the requirements above and if there exists some fresh pair $(z_1,z_2) \in \mathcal{P}$ for which $\{x,  p(z_2)\}$ is claimed by Maker or $z_2$ is a root of some $T_i$ with $i \in [2 \ell]$, Maker claims $xz_2$, if possible.

    In both cases, suppose that $z_2 \in V(T^{(n)}_{j_\ell})$ up to relabeling indices.
    Moreover, we set $T^{(n+1)}_{j_\ell} =T_{j_\ell}\big[V(T^{(n)}_{j_\ell}) \cup \{z_2\}\big]$ and $T^{(n+1)}_{j_i} = T^{(n)}_{j_i}$ for all $i \in [\ell -1]$. Observe that (i)-(iii) hold and start over.

    Now, suppose that the above algorithm has eventually terminated and Maker has claimed $\mathcal{T}' = \{T^{(m)}_{j_1}, \dots, T^{(m)}_{j_\ell}\}$ for some $m \in \N_0$. By (iii), whenever Breaker has claimed either $xr_{2j-1}$ or $xr_{2j}$, Maker must have claimed the other edge, so by (i), Maker has connected $x$ to exactly one vertex from every pair $\{(r_{2j-1}, r_{2j})\}_{j \in [\ell]}$. For each vertex $v$ for which Maker has claimed $xv$, Maker has connected $x$ to at most $m$ of its children by (i).
    Moreover, for each pair of children $(v_1, v_2) \in \mathcal{P}$ for which Breaker has claimed $xv_i$ with $i \in [2]$, Maker must have claimed $xv_{3-i}$, since otherwise Maker has claimed both $xv_3$ and $xv_4$ for some $v_3 \in \lceil v \rceil$ and $(v_3, v_4) \in \mathcal{P}$ by (iii). Therefore, $\mathcal{T}'$ is a family of $\ell$ maximal $m$-ary rooted trees and by (ii), Maker has connected $x$ to $V(\mathcal{T}')$ in exactly $|V(\mathcal{T}')|$ moves.
\end{proof}

\begin{remark}
    \label{rem:connVertexToHalfTrees}
    Observe that, if Maker begins in \Cref{lem:connVertexToHalfTrees}, we get the same conclusion even if we are only given $2 \ell -1$ pairwise disjoint $2m$-ary rooted trees, since Maker can claim $x r_{2 \ell -1}$ in her first move and then play according to the strategy presented in \Cref{lem:connVertexToHalfTrees}.
\end{remark}

The next lemma will be useful in the proof of \Cref{thm:partpattGame}.
\begin{lemma}
    \label{lem:PartPatt:LowerAndUpper}
    We have $\MBpartpatternNumber{k}{H}{\ell t} \leq t \cdot \MBpartpatternNumber{k t}{H}{\ell}$ for each graph $H$ and all $k, \ell, t \in \N$ .
\end{lemma}
\begin{proof}
    Suppose we are given a graph $H$ and $k, \ell, t \in \N$. If $\MBpartpatternNumber{k t}{H}{\ell} = \infty$, we are done, so suppose that $\MBpartpatternNumber{k t}{H}{\ell} \in \N$. In order to show the inequality above, we have to show that Maker has a winning strategy in $\MBpartpattern{k}{H}{t x}{\ell t}$ for $x \coloneqq \MBpartpatternNumber{k t}{H}{\ell}$. To do that, suppose we are given graphs $G^1, \dots, G^{k}$, such that for each $i \in [k]$, $G^i$ consists of pairwise vertex disjoint copies $H^i_1, \dots, H^i_{tx}$ of $H$.
    \begin{figure}[H]
        \centering
        \includegraphics[width=0.60\textwidth]{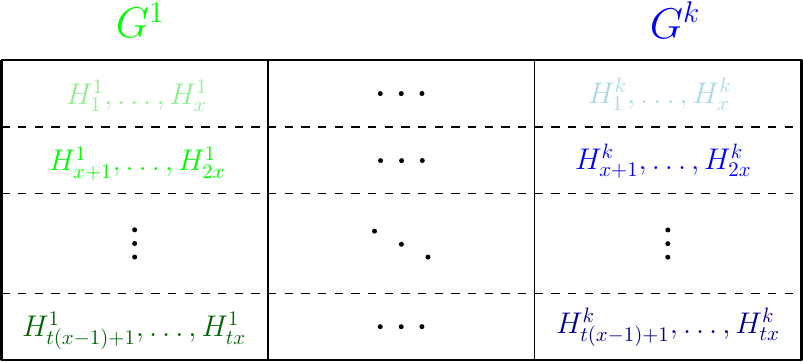}
        \caption{One can think of $G^1, \dots, G^{k}$ as graphs of different (standard) colors. For each color $c$, partition the copies of $H$ in color $c$ into $t$ sets of equal size and recolor them with pairwise different shades of $c$. By definition of $x$, Maker can include exactly $\ell$ copies of $H$ of each of the $k t$ shades into her $\KalephNull$. So she includes $\ell t$ copies of each color in her $\KalephNull$ as desired.}
    \end{figure}
    Let $G^i$ be the disjoint union of $\{\widetilde{G}^i_j\}_{j \in [t]}$, where $\widetilde{G}^i_j = H^i_{(j-1)x +1} \cup \ldots \cup H^i_{j x}$.
    Since each $\widetilde{G}^i_j$ consists of $x$ pairwise disjoint copies of $H$, by choice of $x$, Maker has a strategy to claim a copy $K$ of $\KalephNull$ containing $\ell$ copies of $H$ from every $\widetilde{G}^i_j$. By definition of $\{\widetilde{G}^i_j\}_{j \in [t]}$, $K$ contains $\ell t$ copies of $H$ from each $G^i$ as desired.
\end{proof}

In order to proceed, we need the following immediate consequence of KÃ¶nig's infinity lemma \cite{Koenig1927}.
\begin{lemma}
    \label{lem:KÃ¶nig}
    Every infinite rooted tree $T$ with $d(v) < \alephNull$ for all $v \in V(T)$ contains a rooted ray.
\end{lemma}

In the following proof, we are going to introduce a tool called the \emph{wish function}, due to Arlt, which will be used frequently in the rest of this paper. By using this tool, Arlt came up with a much more transparent proof of \Cref{thm:lucagame} in his bachelor's thesis \cite{LucaBachelor}. Note that a weaker version of \Cref{thm:lucagame} was first proven by Bowler, Emde and Gut \cite{Bowler_2023}.

\begin{theorem}
    \label{thm:lucagame}
    Given $k,s \in \N$, a vertex $r \in B$ and a coloring $c \colon V(B) \to k$, such that $|c^{-1}(i) \cap V(B)| = \alephNull$ for every $i \in k$. In the Maker-Breaker game, in which Maker and Breaker alternately claim edges of $B$, Maker can claim a copy $K$ of $\KalephNull$ in $B$, such that $r \in V(K)$ and $|c^{-1}(i) \cap V(K)| = \alephNull$ for all $i \in k$.
    Moreover, Maker can ensure that after $s (k+1)^{ks+1}$ moves, she has claimed a copy of $K_{ks}$, such that each color appears exactly $s$ times among its vertices and all but $3s (k+1)^{ks+1}$ vertices of each color class are still fresh.
\end{theorem}
We include the full proof here to illustrate the method of wish functions, which will be important later.
\begin{proof}
    We are going to construct a rooted tree recursively as follows. Up to relabeling colors, we have $c(r) = 0$. Set $T_0 = \{r\}$ and suppose for $n \in \N_0$, we have defined a tree $T_n \subseteq B$ rooted at $r$, such that
    \begin{enumerate}
        \item Maker has claimed its comparability graph,
        \item $|V(T_n)| = n+1$,
        \item $ch(v) \leq k+1$ for all $v \in T_n$,
        \item $c(v) = h(v) \Mod{k}$ for all $v \in T_n$.
    \end{enumerate}
    Let $T'_n \subseteq T_n$ be the maximal $(k+1)$-ary subtree of $T_n$. We define a \emph{wish function} $w \colon V(T'_n) \to k$ by setting $w(\ell) = i+1 \Mod{k}$ for every leaf $\ell \in V(T'_n)$ with $c(\ell) = i$ and $i \in k$. For a vertex $v \in V(T'_n)$ for which $w$ is already defined on all of its children, we set $w(v) = j$ for some $j \in k$ for which there exist at least two children in color $j$. By the pigeonhole principle, this yields a well-defined wish function $w$.

    \begin{figure}[htbp]
        \centering
        \includegraphics[width=0.70\textwidth]{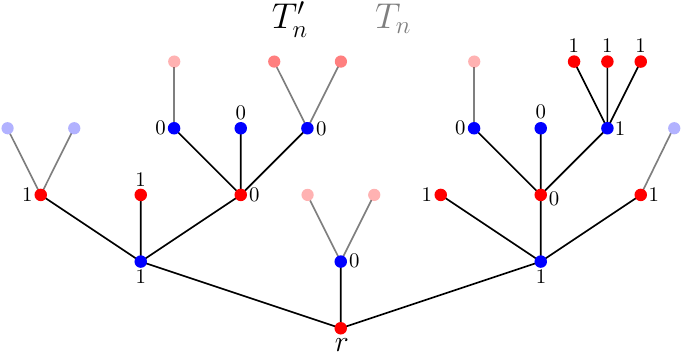}
        \caption{For $k = 2$, let red and blue be color $0$ and $1$, respectively. Next to every vertex of $T'_n$ is written, which color this vertex wishes for. In the figure, we have $w(r) = 1$, since two of its children wish for 1.}
    \end{figure}

    Suppose that $w(r) = i^*$. By the definition of $w$, there exists a maximal binary subtree $T''_n \subseteq T'_n$, such that for all $v \in T''_n$, we have $w(v) = i^*$. Let $x$ be a fresh vertex, such that $c(x) = i^*$. For some leaf $y \in T''_n$, Maker can connect $x$ to all vertices of $P_y$ by \Cref{rem:connVertexToHalfTrees} in exactly $h(y) + 1$ moves. To complete the recursion step, we set $V(T_{n+1}) = V(T_n) \cup \{x\}$ and $E(T_{n+1}) = E(T_n) \cup \{xy\}$.

    By construction of $T_{n+1}$, (i) and (ii) hold.~(iii) holds, since $y$ was a leaf of $T'_n$ and therefore $ch(y) \leq k$ in $T_n$. To verify (iv), observe that by definition of $w$, we have $c(x) = c(y) + 1 = h(y) + 1 = h(x) \Mod{k}$. Finally, let $T_{\alephNull} \coloneqq \bigcup_{n \in \N_0} T_n$. Observe that (iii) still holds for $T_{\alephNull}$. Hence, \Cref{lem:KÃ¶nig} is applicable and we obtain a rooted ray $R \coloneqq r v_1 v_2 \ldots \subseteq T_{\alephNull}$. Let $K$ be the copy of $\KalephNull$ induced by $V(R)$. By (i), Maker has claimed all edges of $K$ and by (iv), we have $|c^{-1}(i) \cap V(K)| = \alephNull$. Since $r \in V(K)$, $K$ is as desired.

    In order to prove that Maker has claimed a copy of $K_{ks}$ after at most $s (k+1)^{ks+1}$ moves, let $b$ be maximal, such that Maker has not claimed a copy of $K_{ks}$ in the comparability graph of $T_b$, which contains exactly $s$ vertices of each color. By (iii), one can equivalently ask for the maximal $b$ such that there exists no rooted path of length $ks$ in $T_b$. So $T_b \subseteq T(k+1, ks-1)$.
    Let $\hat{T}_{i}$ be the comparability graph of $T_{i}$.
    Thus, Maker has first claimed the desired $K_{ks}$ at the point where she has claimed precisely the edges of $\hat{T}_{b+1}$. We conclude that
    \begin{align*}
        e(\hat{T}_{b+1}) & \leq ks + e(\hat{T}_b)
        \leq ks + \sum_{i = 1}^{ks-1} i(k+1)^i \\
        & \leq ks \left( 1 + \sum_{i = 1}^{ks-1} (k+1)^i \right)
        = ks \left( 1 + \frac{(k+1)^{ks} -1}{(k+1) -1}  \right)
        \leq s \cdot (k+1)^{ks+1}.
    \end{align*}
    Moreover, since Breaker has claimed edges incident to at most $2s \cdot (k+1)^{ks+1}$ vertices in the meantime, all but at most $2s \cdot (k+1)^{ks+1} + |V(T_{b+1})| \leq 3s \cdot (k+1)^{ks+1}$ vertices are still fresh after Maker has claimed a copy of $K_{ks}$. This completes the proof.
\end{proof}
Observe that in the proof of \Cref{thm:lucagame}, it was essential to \emph{satisfy the wish} of $r$ with respect to $w$, i.e. to connect a fresh vertex of color $i^*$ to $r$ in order to build $T_{n+1}$ from $T_n$. Thus, the name \emph{wish function}.

\section{Maker's main strategy and the equivalence of games}
\label{ch:FinInfEquiv}
In this section, we first formalize the game played in \Cref{thm:(part)pattStar}, then we introduce a finite game and show that it is equivalent to the game played in \Cref{thm:(part)pattStar}. Finally, we give a sufficient condition for Maker to have a winning strategy in the finite game and thereby in \Cref{thm:(part)pattStar}.

\begin{definition}
    Given the board $B$, integers $C, s, t \in \N_0$ and pairwise vertex disjoint graphs $G_1, \dots, G_{s+t} \subseteq B$, such that each of $G_1, \dots, G_s$ is a copy of $G_{\text{inf}}^{C}$ and each of $G_{s+1}, \dots, G_{s+t}$ is a copy of $G_{\text{star}}$. Maker wins the game $\MBAscInfStars{C}{s}{t}$, if she can claim a copy $K$ of $\KalephNull$, such that $G_i \cap K$ induces a copy of $\SalephNull$ for each $i \in [s]$ and $G_j \cap K$ induces a copy of $G_{\text{star}}$ for each $j \in \{s+1, \dots, s+t\}$.
\end{definition}
Now, we introduce a finite game, which we will prove to be equivalent to the game above.
To do that, we need to introduce the following auxiliary game.

\begin{definition}
    \label{def:auxGame}
    Let $\mathcal{G}$ be a finite hypergraph. In the game $\MBauxM{\mathcal{G}}{1}$ (respectively $\MBauxB{\mathcal{G}}{1}$), Maker (respectively Breaker) begins and Maker and Breaker alternately claim vertices of $\mathcal{G}$. Maker wins if she has a strategy to claim all vertices of an edge $e \in E(\mathcal{G})$. Otherwise, Breaker wins.
    In the game $\MBauxM{\mathcal{G}}{2}$, Maker and Breaker alternately claim vertices of $\mathcal{G}$. Maker begins and she wins, if she can ensure that amongst her claimed vertices there is a vertex set $X \subseteq V(\mathcal{G})$, such that she has a winning strategy in $\MBauxB{\mathcal{G}[X]}{1}$.
\end{definition}
For a fixed $v \notin V(\mathcal{G})$, we define the games $\MBauxVertexM{\mathcal{G}}{1}{v}, \MBauxVertexB{\mathcal{G}}{1}{v}, \MBauxVertexM{\mathcal{G}}{2}{v}$ similarly to their respective vertex version with the only difference being that Maker and Breaker alternately claim edges $vw$ with $w \in V(\mathcal{G})$ instead of vertices.
It is easy to see that, $\MBauxVertexM{\mathcal{G}}{k}{v}$ is equivalent to $\MBauxM{\mathcal{G}}{k}$ for $k \in [2]$ and $\MBauxVertexB{\mathcal{G}}{1}{v}$ is equivalent to $\MBauxB{\mathcal{G}}{1}$.

\begin{remark}
    \label{rem:passingAllowed}
    Observe that, if Breaker has a winning strategy in $\MBauxM{\mathcal{G}}{2}$, he also has a winning strategy in $\MBauxM{\mathcal{G}}{2}$ with the additional rule that both players are allowed to pass. This holds, since Breaker can always assume that Maker just claims any vertex whenever she passes or claims a vertex that Breaker already assumed Maker has claimed.
    If by Breaker's assumption, all vertices of $\mathcal{G}$ have been claimed, he claims any vertex. By definition of $\MBauxM{\mathcal{G}}{2}$, Breaker has been able to claim a vertex set $X$, such that Maker does not have a winning strategy in $\MBauxB{\mathcal{G}[Y]}{1}$ for $Y \coloneqq V(\mathcal{G}) \backslash X$. Hence, he won the game $\MBauxM{\mathcal{G}}{2}$ with the additional passing rule.
    Similarly for $\MBauxM{\mathcal{G}}{1}$, $\MBauxB{\mathcal{G}}{1}$ and if Maker has a winning strategy in any of the three games.
\end{remark}

Now, we define the finite game for which we will prove that it is equivalent to $\MBAscInfStars{C}{s}{t}$. For that recall \Cref{def:KCs}.
\begin{definition}
    \label{def:MBfin}
    Given $C, s, t \in \N_0$. Suppose that Maker and Breaker alternately claim edges of $E(K_C^s)$. Maker begins. If she can claim an edge set $H \subseteq E(K_C^s)$, such that for every coloring $c \colon E(K_C^{(s)}(H)) \to [Cs+t]$ with $r_j \in e$ for all $e \in c^{-1}(j)$ and $j \in [Cs]$, there exists an $i \in [Cs+t]$, such that Maker has a winning strategy in $\MBauxM{c^{-1}(i)}{1}$, if $i \in [Cs]$ or in $\MBauxM{c^{-1}(i)}{2}$, if $i \in \{Cs+1, \dots, Cs + t\}$, we say that Maker has a winning strategy in $\MBfin{C}{s}{t}$. Otherwise, we say that Breaker has a winning strategy in $\MBfin{C}{s}{t}$.
\end{definition}

We now prove that the two games $\MBfin{C}{s}{t}$ and $\MBAscInfStars{C}{s}{t}$ are equivalent by proving both directions separately.

\begin{theorem}
    \label{thm:eq:Maker}
    Given $C, s, t \in \N_0$. If Maker has a winning strategy in $\MBfin{C}{s}{t}$, she has a winning strategy in $\MBAscInfStars{C}{s}{t}$.
\end{theorem}
\begin{proof}
    Suppose we are given pairwise vertex disjoint graphs $G_1, \dots, G_{s+t} \subseteq B$, such that for each $i \in [s]$, $G_i$ consists of infinite stars $S_{C(i-1) + 1}, \dots, S_{Ci}$ rooted at $r_{C(i-1) +1}, \dots, r_{Ci}$, respectively and for $i \in \{s+1, \dots, s+t\}$, $G_i$ is the union of an ascending family of finite stars $\{\hat{S}_j^{i-s}\}_{j \in \N}$ with roots $\{\hat{r}_j^{i-s}\}_{j \in \N}$, respectively and $|\hat{S}_j^{i-s}| = j$. Let $K_C^s$ be the graph induced by the partition classes $\{r_{C(i-1) +1}, \dots, r_{Ci}\}$ with $i \in [s]$.
    For convenience, set $\tau = s+3t+1$ and let $\alpha = (3Cs + 4)(s+t)$. Maker begins by claiming $A \subseteq E(K_C^s)$ as desired by her strategy in $\MBfin{C}{s}{t}$. For each hyperedge $e \in E(K_C^{(s)}(A))$, we fix an enumeration of its vertices $e = \{r_{j_1(e)}, \dots, r_{j_s(e)}\}$ for the rest of the proof. In addition for each $e \in E(K_C^{(s)}(A))$, we fix a set of pairwise different fresh vertices $R_e \coloneqq \{r_{(e,1)}, \dots, r_{(e,\alpha)}\}$, such that for $e \neq e'$, we have $R_e \cap R_{e'} = \emptyset$. Those vertices will serve as roots of some auxiliary trees. Maker is going to proceed recursively on $n$.
    \begin{figure}[H]
        \centering
        \includegraphics[width=0.60\textwidth]{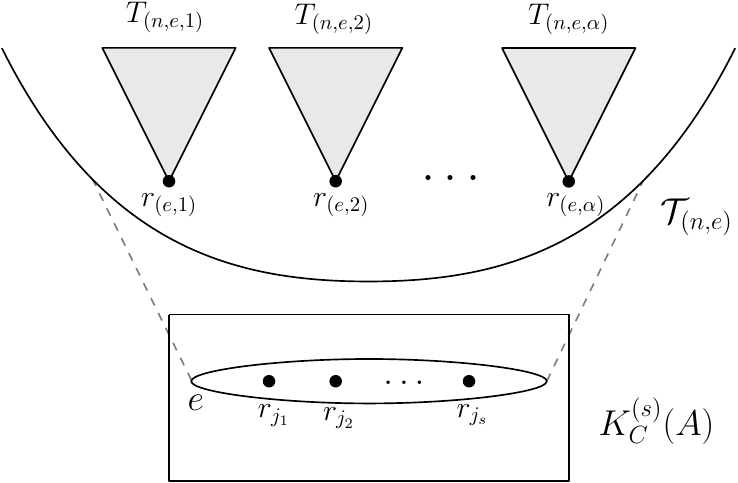}
        \caption{In the figure above is depicted what Maker has constructed after the $n^\text{th}$ recursion step. Each bag $\mathcal{T}_{(n,e)}$ contains a fixed amount of rooted trees, such that Maker has claimed each of their comparability graphs and connected every vertex of $\mathcal{T}_{(n,e)}$ except for the roots to all vertices of $e$. In the $(n+1)^\text{st}$ step, Maker is going to extend one of the trees $T_{(n,e,i)}$ for some $e \in E(K_C^{(s)}(A))$.}
        \label{pic:RecursionStepEquiv1}
    \end{figure}
    After the $n^\text{th}$ recursion step, Maker will have constructed a set $\mathcal{T}_{(n,e)}$ of $\alpha$ trees for each $e \in E(K_C^{(s)}(A))$ as in \Cref{pic:RecursionStepEquiv1}. Each tree $T \in \mathcal{T}_{(n,e)}$ will consist of leaves of the infinite stars $\{S_i\}_{i \in [Cs]}$ and finite subgraphs of $\{\hat{S}_j^i\}_{i \in [t], j \in \N}$, such that each such subgraph spans a $T(\tau, h_S)$ rooted at $\hat{r}_j^i$ for pairwise different $h_S$ (see \Cref{pic:RecursionStepEquiv2}).
    Moreover, she will build $T$ in such a way, that for any rooted path $P \in T$, the vertices of height $1, \dots, s$ will be leaves of infinite stars of $G_1, \dots, G_s$, respectively. After this will come the root of some $\hat{S}_{j_1}^1$ followed by some leaves of that star. Then, the root of some $\hat{S}_{j_1}^1$ followed by some of its leaves and so on until we have $\hat{r}_{j_t}^t$ for some $j_t$, followed by some leaves of $\hat{S}_{j_t}^t$. Thereafter the cycle begins again with leaves of infinite stars of $G_1, \dots, G_s$ as in \Cref{pic:RecursionStepEquiv2}.

    Formally, for all $e \in E(K_C^{(s)}(A))$ and $i \in [\alpha]$, we define the rooted tree $T_{(0, e, i)} \coloneqq \{r_{(e,i)}\}$ and set $\mathcal{T}_{(0,e)} \coloneqq \{T_{(0,e,1)}, \dots, T_{(0,e, \alpha)}\}$ for each $e \in E(K_C^{(s)}(A))$.
    \begin{figure}[H]
        \centering
        \includegraphics[width=0.72\textwidth]{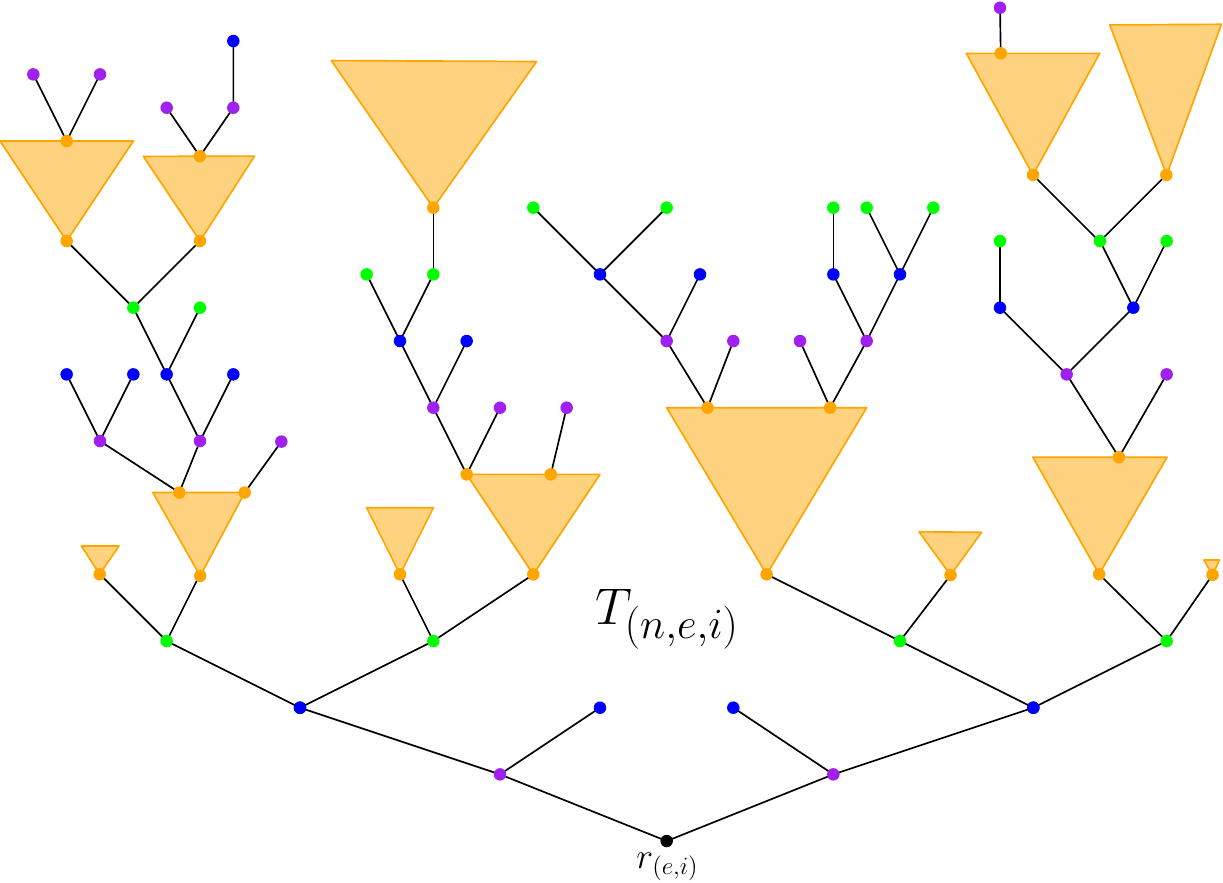}
        \caption{
            The figure above depicts a binary subtree of $T_{(n,e,i)}$ in the case of $s=3$ and $t=1$. Purple, blue and green vertices correspond to vertices from $G_1, G_2$ and $G_3$, respectively and each orange triangle is a copy of $T(\tau, h)$ for pairwise different $h$ with root $\hat{r}_b^1 \in G_4$, such that all vertices of an orange triangle are contained in the same finite star of $G_4$.
            For each rooted path, Maker must have an `ascending color pattern', which means in our case that with respect to the tree order of $T_{(n,e,i)}$, every purple vertex is followed by a blue vertex, then a green vertex, thereafter parts of an orange triangle, then a purple vertex again and so on.
            In the $(n+1)^\text{st}$ recursion step, Maker is going to attach either a leaf of an infinite star of $G_1, G_2$ or $G_3$ or a finite star of $G_4$ to $T_{(n,e,i)}$, such that the ascending color pattern with respect to the tree order is preserved.
        }
        \label{pic:RecursionStepEquiv2}
    \end{figure}

    Maker will now recursively build for each $n \in \N_0$ and each edge $e \in E(K_C^{(s)}(A))$ a family of $\alpha$ trees $\mathcal{T}_{(n,e)} = \{T_{(n,e,1)}, \dots, T_{(n,e, \alpha)}\}$ rooted at $r_{(e,1)}, \dots, r_{(e,\alpha)}$, respectively, such that
    \begin{enumerate}
        \item for each $T \in \mathcal{T}_{(n,e)}$, Maker has claimed the comparability graph of $T$,
        \item $|\bigcup_{e} V(\mathcal{T}_{(n,e)})| > |\bigcup_{e} V(\mathcal{T}_{(n-1,e)})|$,
        \item for all $v \in V(\mathcal{T}_{(n,e)}) \setminus R_e$, $v$ is connected to all vertices of $e$ and $ch(v) \leq \tau$,
        \item for all $i \in [s]$ and $v \in V(\mathcal{T}_{(n,e)}) \cap V(S_{j_i(e)})$, we have $p(v) \in V(G_{s+t}) \cup R_e$, if $i = 1$ and $p(v) \in V(S_{j_{i-1}(e)})$, otherwise,
        \item  for all $i \in [t]$ and $v \in V(\mathcal{T}_{(n,e)}) \cap V(\hat{S}_k^i)$ with $k \in \N$, we have
              \begin{itemize}
                  \item $p(v) \in V(S_{j_s(e)})$, if $v = \hat{r}_k^1$,
                  \item $p(v) \in V(G_{s+i-1})$, if $v = \hat{r}_k^i$ for $i > 1$,
                  \item $p(v) \in V(\hat{S}_k^i)$, otherwise.
              \end{itemize}
              Moreover, for every root $\hat{r}_k^i \in V(\mathcal{T}_{(n,e)})$, let $h = h(\hat{r}_k^i)$ be maximal, such that $U^*_h(\hat{r}_k^i) \subseteq V(\hat{S}_k^i)$. Then, we have $h(\hat{r}) \leq n$ for all $\hat{r} \in \{\hat{r}_j^i\}_{j \in \N} \cap V(\mathcal{T}_{(n,e)})$ and $h(\hat{r}) \neq h(\hat{r}')$, if $\hat{r} \neq \hat{r}'$.
    \end{enumerate}
    For the execution of the recursion step, let $h^* \in \N$ be such that each tree in $\bigcup_{e} \mathcal{T}_{(n,e)}$ is of height less than $h^*$.
    For all $e \in E(K_C^{(s)}(A))$ and $k \in [\alpha]$, we define $T_{(e,k)}'$ to be the maximal $\tau$-ary subtree of $T_{(n,e,k)}$ rooted at $r_{(e,k)}$. We define an auxiliary wish function $w_{(e,k)} \colon V(T_{(e,k)}') \to \{j_1(e), \dots, j_s(e), Cs+1, \dots, Cs+t \}$, by setting for a leaf $\ell$ of $T_{(e,k)}'$:
    \[
        w_{(e,k)}(\ell) =
        \begin{cases*}
            j_1(e)     & \text{if $\ell \in V(G_{s+t}) \cup R_e$}        \\
            j_{i+1}(e) & \text{if $\ell \in V(S_{j_i(e)})$ and $i < s$,} \\
            Cs+1       & \text{if $\ell \in V(S_{j_s(e)})$,}             \\
            Cs+i+1     & \text{if $\ell \in V(G_{s+i})$ and $i < t$.}    
        \end{cases*}
    \]
    For a vertex $v \in V(T_{(e,k)}')$ for which $w_{(e,k)}$ is already defined on all of its children, we set $w_{(e,k)}(v)$ to be a color for which there exist at least two children in one color from $\{j_1(e), \dots, j_s(e)\}$ or at least four children in one color from $\{Cs+1, \dots, Cs+t\}$. By the pigeonhole principle, this yields a well-defined wish function $w_{(e,k)}$.

    We are going to define a coloring $c \colon E(K^{(s)}_C(A)) \to [Cs+t]$ as follows. By the pigeonhole principle, there must exist $i^* \in \{j_1(e), \dots, j_s(e), Cs+1, \dots, Cs+t\}$ and a set $Q_e \subseteq [\alpha]$ of size $\frac{\alpha}{s+t}$, such that $w_{(e,q)}(r_{(e,q)}) = i^*$ for all $q \in Q_e$. For any such $Q_e$ and $i^*$, we set $c(e) = i^*$.
    Since $c(e) \in \{j_1(e), \dots, j_s(e), Cs+1, \dots, Cs+t\}$, we must have $r_{i^*} \in e$ in the case where $c(e) = i^*$ for $i^* \in [Cs]$, so $c$ satisfies the property of the coloring from the game $\MBfin{C}{s}{t}$.

    Let $k \in [Cs+t]$ be, such that Maker has a winning strategy in $\MBauxM{c^{-1}(k)}{1}$ or $\MBauxM{c^{-1}(k)}{2}$ depending on the value of $k$.

    \vspace{0.6em}
    \underline{Case 1}: $k \in [Cs]$
    \vspace{-0.8em}
    \begin{caseproof}{1}
        Then Maker has a winning strategy in $\MBauxM{c^{-1}(k)}{1}$ by \Cref{def:MBfin}. Let $x$ be a fresh leaf of $S_k$. Since Maker has a winning strategy in $\MBauxM{c^{-1}(k)}{1}$, she also has a winning strategy in $\MBauxVertexM{c^{-1}(k)}{1}{x}$. Therefore, Maker can (and does) connect $x$ to an edge $\tilde{e} \in E(K^{(s)}_C(A))$ with $c(\tilde{e}) = k$ in at most $Cs$ moves. 
        By the pigeonhole principle, there exists $Q_{\tilde{e}} \subseteq [\alpha]$ of size $\frac{\alpha}{s+t}$, such that $w_{(\tilde{e},q)}(r_{(\tilde{e},q)}) = k$ for all $q \in Q_{\tilde{e}}$. 
        Since $|Q_{\tilde{e}}| > Cs+1$, there exists $\tilde{q} \in Q_{\tilde{e}}$, such that no edge of $E(\{x\}, T'_{(\tilde{e},\tilde{q})})$ has been claimed.
        By construction of $w_{(\tilde{e},\tilde{q})}$, there exists a maximal binary subtree $T''_{(\tilde{e},\tilde{q})} \subseteq T'_{(\tilde{e},\tilde{q})}$, such that for all $v \in V(T''_{(\tilde{e},\tilde{q})})$, we have $w_{(\tilde{e},\tilde{q})}(v) = k$.
        By \Cref{rem:connVertexToHalfTrees}, Maker can (and does) connect $x$ to a maximal rooted path of $T''_{(\tilde{e},\tilde{q})}$ with end vertex $y$. Set
        \begin{align*}
            V(T_{(n + 1, \tilde{e},\tilde{q})}) & = V(T_{(n, \tilde{e},\tilde{q})}) \cup \{x\},    \\
            E(T_{(n + 1, \tilde{e},\tilde{q})}) & = E(T_{(n, \tilde{e},\tilde{q})}) \cup \{y, x\},
        \end{align*}
        and $T_{(n +1, e, j)} = T_{(n, e, j)}$ for all $(e,j) \in E(K^{(s)}_C(A)) \times [\alpha] \setminus \{(\tilde{e}, \tilde{q})\}$ and start over. (i) and (ii) hold by construction. For (iii), firstly Maker has connected $x$ to $\tilde{e}$ by definition of $\tilde{e}$ and secondly, observe that $y$ was also a leaf of $T'_{(\tilde{e},\tilde{q})}$ and by construction of $T'_{(\tilde{e},\tilde{q})}$, we have $ch(y) < \tau$ in $T_{(n, \tilde{e},\tilde{q})}$. In order to verify (iv), observe that we have $k = j_i(\tilde{e})$ for some $i \in [s]$. If $i = 1$, we have $p(x) \in V(G_{s+t}) \cup R_{\tilde{e}}$ and if $i >1$, we have $p(x) \in S_{j_{i-1}(\tilde{e})}$ by definition of $w_{(\tilde{e},\tilde{q})}(y)$. Since (v) holds trivially, this completes the recursion step in Case 1.
    \end{caseproof}

    \underline{Case 2}: $k \in \{Cs+1, \dots, Cs+t\}$
    \vspace{-0.8em}
    \begin{caseproof}{2}
        In this case, Maker has a winning strategy in $\MBauxM{c^{-1}(k)}{2}$ by \Cref{def:MBfin}. Let
        \begin{align*}
            k' = k - Cs, \hspace*{1cm}
            \delta = Cs + h^* + \tau +1, \hspace*{1cm}
            \gamma = \delta (\tau +1)^n,
        \end{align*}
        Maker selects a fresh star $\hat{S}_{\beta}^{k'}$ rooted at $\hat{r}_{\beta}^{k'}$ with $\beta \geq C^s \alpha 2^{h^*} (2(Cs+1+ h^* +n)\gamma +2)$. Since Maker has a winning strategy in $\MBauxM{c^{-1}(k)}{2}$, she also has a winning strategy in $\MBauxVertexM{c^{-1}(k)}{2}{\hat{r}_{\beta}^{k'}}$. Therefore, she can (and does) connect $\hat{r}_{\beta}^{k'}$ to a subset $X \subseteq \{r_1, \dots, r_{Cs}\}$ in at most $Cs$ moves, such that she has a winning strategy in $\MBauxB{c^{-1}(k)[X]}{1}$. Let $A^*$ be the edge set of $c^{-1}(k)[X]$. By definition of $c$ for each $e \in A^*$, there exists a set $Q_e \subseteq [\alpha]$ of size $\frac{\alpha}{s+t} -Cs$, such that $w_{(e,q)}(r_{(e,q)}) = k$ for all $q \in Q_e$ and no edge of $E(\{\hat{r}_{\beta}^{k'}\}, T'_{(e,q)})$ has been claimed yet. By definition of $w_{(e,q)}$, there exists a maximal 4-ary rooted subtree $T''_{(e,q)} \subseteq T'_{(e,q)}$ for each $e \in A^*$ and $q \in Q_e$, such that for all $v \in V(T''_{(e,q)})$, we have $w_{(e,q)}(v) = k$. 

        Now, we apply \Cref{lem:connVertexToHalfTrees} on $\hat{r}_{\beta}^{k'}$ and $\mathcal{T} = \{T''_{(e,q)}\}_{(e,q) \in \bigcup_{e \in A^*} \{e\} \times Q_e}$ with the enumeration of trees, where all the $T''_{(e,q)}$ with $q \in Q_e$ appear as a consecutive block of even size for each $e \in A^*$ and so if $T_{i^*} = T''_{(e',q')}$ and $T_{i^*+1} = T''_{(e'',q'')}$, we have $e' = e''$ for each odd $i^*$.
        By \Cref{lem:connVertexToHalfTrees}, suppose that for each $e \in A^*$, Maker has connected $\hat{r}_{\beta}^{k'}$ to maximal binary subtrees $T'''_{(e,q')} \subseteq T''_{(e,q')}$ with $q' \in Q'_e$ for $Q'_e \subseteq Q_e$ being of size exactly $\frac{|Q_e|}{2}$. Moreover, Maker can do so using at most $2^{h^*}|A^*||Q'_e| \leq 2^{h^*} C^s \alpha$ moves in total.
        \begin{figure}[H]
            \centering
            \includegraphics[width=0.47\textwidth]{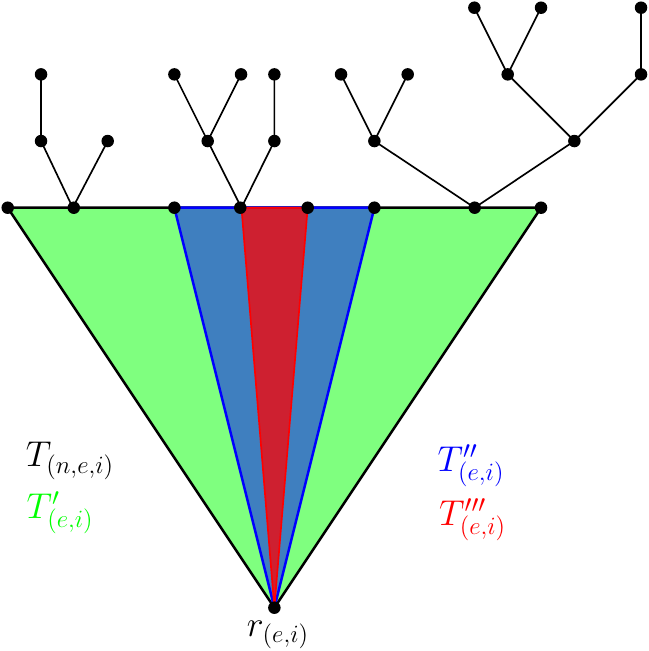}
            \caption{
                $T_{(n,e,i)}$: The tree given by the recursive construction. \newline
                $T'_{(e,i)}$: The maximal $\tau$-ary rooted subtree of $T_{(n,e,i)}$. \newline
                $T''_{(e,i)}$: The maximal 4-ary rooted subtree of $T'_{(e,i)}$, such that $w_{(e,i)}$ is constant on $T''_{(e,i)}$. \newline
                $T'''_{(e,i)}$: The maximal binary rooted subtree of $T''_{(e,i)}$ to which Maker has connected a root of a finite star.
            }
            \label{pic:TheDifferentT}
        \end{figure}
        For each $e \in A^*$, $q \in Q'_e$ and every leaf $\ell \in V(T'''_{(e,q)})$, let $H^0_{(e, q, \ell, 1)} = \ldots = H^0_{(e, q,\ell, \delta)} = \emptyset$ and suppose that Maker has recursively constructed rooted auxiliary trees $H^m_{(e, q, \ell, 1)}, \dots, H^m_{(e, q, \ell, \delta)}$ for some $m \in \N_0$, such that for each $i \in [\delta]$:
        \begin{enumerate}[label=(\Roman*)]
            \item Maker has claimed the comparability graph of $H^m_{(e, q, \ell, i)}$,
            \item each $v \in H^m_{(e, q, \ell, i)}$ is connected to all vertices of $P_\ell \cup e \cup \{\hat{r}_{\beta}^{k'}\}$, for $P_\ell \subseteq T'''_{(e,q')}$ being the unique path from $r_{(e,q')}$ to $\ell$.
        \end{enumerate}

        For the recursion step on $m$, we define a bag $B_{(e, q, \ell)} \coloneqq \{H^m_{(e, q, \ell, 1)}, \dots, H^m_{(e, q, \ell, \delta)}\}$ for each $e \in A^*$, $q \in Q'_e$ and every leaf $\ell \in V(T'''_{(e,q)})$.

        As long as there exists a fresh leaf $v \in \hat{S}_{\beta}^{k'}$, by definition of $\MBauxVertexB{c^{-1}(k)}{1}{v}$, Maker claims $v \hat{r}_{\beta}^{k'}$ first and then she can (and does) connect $v$ to an edge $e' \in A^*$ in at most $Cs +1$ moves. Since $Q'_{e'}$ is of size at least $Cs+2$, there exists a $q' \in Q'_{e'}$, such that no edge of $E(\{v\}, T'''_{(e',q')})$ has been claimed. By \Cref{rem:connVertexToHalfTrees}, Maker can (and does) connect $v$ to a maximal path $P_{\ell'}$ of $T'''_{(e',q')}$ in at most $h^*$ moves. If there exists $j \in [\delta]$, such that $H^m_{(e', q', \ell', j)} = \emptyset$, let $H^{m+1}_{(e', q', \ell', j)} \coloneqq \{v\}$, set $v$ to be its root and set $H^{m+1}_{(e, q, \ell, i)} \coloneqq H^{m}_{(e, q, \ell, i)}$ for all other auxiliary trees. Observe that (I) and (II) hold and start over.
        \begin{figure}[H]
            \centering
            \includegraphics[width=0.55\textwidth]{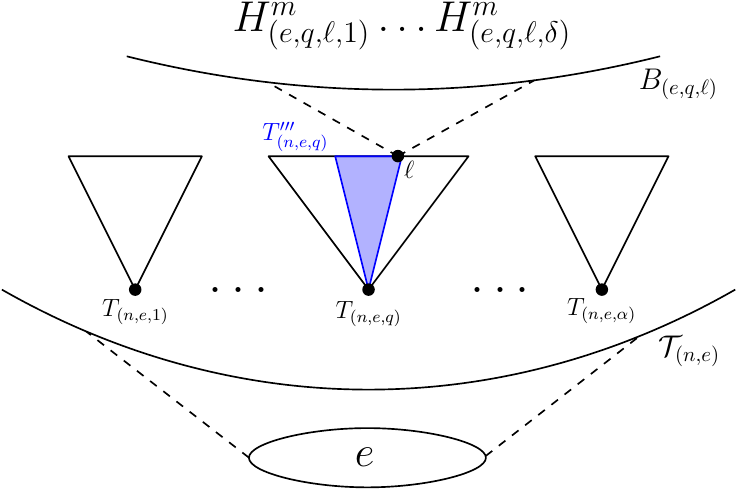}
            \caption{
                The relation of $H^m_{(e, q,\ell, i)}$ to $e, q$ and $\ell$.
            }
        \end{figure}
        Otherwise, let $\tilde{r}_1, \dots, \tilde{r}_\delta$ be the roots of $H^m_{(e', q', \ell', 1)}, \dots, H^m_{(e', q', \ell', \delta)}$, respectively. After Maker has connected $v$ to $P_{\ell'}$, there must exist a set $Z \subseteq [\delta]$ of size at least $\tau$, such that for all $z \in Z$ no edge of $E(\{v\}, H^m_{(e', q', \ell', z)})$ has been claimed. Maker now claims $v \tilde{r}_{z'}$ for $z' \in Z$, such that $|V(H^m_{(e', q', \ell', z')})|$ is minimal. Suppose that Maker has connected $v$ to a rooted path $P_i \subseteq H^m_{(e', q', \ell', z')}$ of length $i$ with end vertex $y$, such that after claiming $vy$, no edge of $E(\{v\}, \lfloor y \rfloor)$ has been claimed.

        If $i < n-1$ and $ch(y) = \tau + 1$, let $y_1, \dots, y_{\tau + 1}$ be the children of $y$. Maker now claims $vy_j$, such that after her move no edge of $E(\{v\}, \lfloor y_j \rfloor)$ has been claimed and among the remaining choices, such that $U(y_j)$ is minimal for $j \in [\tau+1]$. We define $P_{i+1}$ by adding $y_j$ to its vertex set and $yy_j$ to its edge set.
        Otherwise, define $H^{m+1}_{(e', q', \ell', z')}$ by adding $v$ to its vertex set and $vy$ to its edge set. Set $H^{m+1}_{(e, q, \ell, a)} = H^{m}_{(e, q, \ell, a)}$ for all $a \neq z'$ and observe that (I) and (II) both hold.
        This completes the recursion on $m$.

        Now suppose there exists no fresh leaf $v \in \hat{S}_{\beta}^{k'}$ anymore. Since the total number of leaves in $\bigcup_{(e,q)} T'''_{(e,q)}$ is at most $C^s \alpha 2^{h^*}$, by definition of $\beta$, there exists a bag $B_{(e', q', \ell')}$ with $|V(B_{(e', q', \ell')})| \geq \gamma$ for $e' \in A^*$, $q' \in Q'_{e'}$ and a leaf $\ell' \in V(T'''_{(e',q')})$.
        Consider $H^\gamma_{(e', q', \ell', 1)}, \dots, H^\gamma_{(e', q', \ell', \delta)}$, such that $|V(H^\gamma_{(e', q', \ell', 1)})| \geq \ldots \geq |V(H^\gamma_{(e', q', \ell', \delta)})|$ up to relabeling. 
        By definition of Maker's strategy, we have $|V(H^\gamma_{(e', q', \ell', 1)})| \leq |V(H^\gamma_{(e', q', \ell', \tau)})| +1$. For each $i \in [\tau]$, we have $|V(H^\gamma_{(e', q', \ell', i)})| \geq (\tau+1)^n$ by the pigeonhole principle. 
        Let $j$ be maximal, such that there exists a copy $T^j_i$ of $T(\tau, j)$ in $H^\gamma_{(e', q', \ell', i)}$ with root $\tilde{r}_i$, such that each leaf $\hat{v}$ of $T^j_i$ satisfies $U(\hat{v}) \geq (\tau +1)^{n-j}$.
        By choice of $\gamma$, we have $j \geq 0$.

        As long as $j < n$, for each leaf $\hat{v}$ of $T_i^j$ there exist children $\hat{v}_1, \dots, \hat{v}_{\tau+1}$ enumerated decreasingly with respect to $U(\hat{v}_1), \dots, U(\hat{v}_{\tau+1})$, such that $U(\hat{v}_{\tau}) \leq U(\hat{v}_1) +1$ by definition of Maker's strategy. Hence, $U(\hat{v}_1), \dots, U(\hat{v}_{\tau}) \geq (\tau+1)^{n-j-1}$. Define $T^{j+1}_i$ by adding $\hat{v}_1, \dots, \hat{v}_{\tau}$ to the vertex set of $T^j_i$ and $\hat{v}\hat{v}_1, \dots, \hat{v}\hat{v}_{\tau}$ to the edge set of $T^j_i$ for each such $\hat{v}$. Hence, there exists a copy $T^n_i$ of $T(\tau, n)$ in $H^\gamma_{(e', q', \ell', i)}$ for each $i \in [\tau]$, such that (I) and (II) hold for $T^n_i$. Recall \Cref{def:attach}. We define $\tilde{T}$ by attaching $T_1^n, \dots, T_{\tau}^n$ to the rooted tree $\{\hat{r}_{\beta}^{k'}\}$. Finally, we let $T_{(n+1, e', q')}$ be the tree where we attach the rooted tree $\tilde{T}$ to $\ell' \in T_{(n, e', q')}$. For all $(e,q) \in E(K_C^{(s)}(A)) \times [\alpha] \setminus \{(e', q')\}$, we set $T_{(n+1, e, q)} \coloneqq T_{(n, e, q)}$. All we have to show in order to complete the recursion step is that (i)-(v) hold for $T_{(n+1, e', q')}$.

        In order to show that (i) holds, it is sufficient to check whether Maker has connected every $v \in V(\tilde{T})$ to $\lceil v \rceil$. Observe that, $\hat{r}_{\beta}^{k'}$ is connected to $P_{\ell'} = \lceil \hat{r}_{\beta}^{k'} \rceil$, since $P_{\ell'} \subseteq T'''_{(e', q')}$.
        For each $v \in T_i^n$ with $i \in [\tau]$, we have that $v$ is connected to $P_{\ell'} \cup \{\hat{r}_{\beta}^{k'}\}$ by (II) and Maker has claimed the comparability graph of $T_i^n$ by (I). Hence, $v$ is connected to $\lceil v \rceil$ for all $v \in V(\tilde{T})$. Since, we have attached $\tilde{T}$ to $T_{(n, e', q')}$ in the recursion step, (ii) holds.
        For (iii), observe that $\hat{r}_{\beta}^{k'}$ is connected to $e'$, since $e' \in A^*$ and each $v \in \bigcup_{i \in [\tau]} V(T^n_i)$ is also connected to $e'$ by (II). Moreover, observe that $\tilde{T}$ is $\tau$-ary by construction and since $\ell'$ was a leaf in $T'''_{(e', q')}$, it was also a leaf in $T'_{(e', q')}$, so $ch(\ell')< \tau$ in $T_{(n, e', q')}$ by definition of $T'_{(e', q')}$. Therefore, we have $ch(v) \leq \tau$ for all $v \in T_{(n+1, e', q')}$. For (iv), there is nothing to show.

        For (v), observe that $w_{(e',q')}(\ell') = k = Cs + k'$, since $\ell' \in V(T''_{(e',q')})$. By definition of $w_{(e',q')}$, this implies $\ell' \in V(S_{j_s(e')})$, if $k = Cs+ 1$ and $\ell' \in V(G_{s + k' - 1})$ otherwise. Hence, $p(\hat{r}_{\beta}^{k'})$ is as desired. Moreover, $p(v) \in \hat{S}_{\beta}^{k'}$ for any $v \in V(\tilde{T}) \setminus \{\hat{r}_{\beta}^{k'}\}$.
        Lastly, $h(\hat{r}_{\beta}^{k'}) = n+1$ is maximal, such that $U^*_{n+1}(\hat{r}_{\beta}^{k'}) \subseteq \hat{S}_{\beta}^{k'}$ by construction. Since any root $\hat{r} \in \{\hat{r}_j^i\}_{i \in [t], j \in \N} \cap \bigcup_e \mathcal{T}_{(n,e)}$ satisfies $h(\hat{r}) \leq n$ by assumption, this completes the recursion step in the second case.
    \end{caseproof}

    For each $e \in E(K_C^{(s)}(A))$ and $a \in [\alpha]$, let $T_{(\alephNull,e,a)} = \bigcup_{n \in \N} T_{(n,e,a)}$. By (ii), there exists $a^* \in [\alpha]$ and $e^* \coloneqq \{r_{j_1}, \dots, r_{j_s}\} \in E(K_C^{(s)}(A))$, such that $|V(T_{(\alephNull,e^*,a^*)})| = \alephNull$. By (iii), $ch(v) \leq \tau$ for all $v \in V(T_{(\alephNull,e^*,a^*)})$, so by \Cref{lem:KÃ¶nig}, there exists a rooted ray $R \subseteq T_{(\alephNull,e^*,a^*)}$. By (i), (iii) and since $e^* \in E(K_C^{(s)}(A))$, Maker has claimed all edges of the $\KalephNull$ induced by $V(R \setminus \{r_{(e^*, a^*)}\}) \cup e^*$, which we denote with $K$.
    By (iv) and (v), we have $|V(R) \cap V(S_{j_i(e^*)})| = \alephNull$ for each $i \in [s]$. Hence, $K \cap G_i$ is a copy of $\SalephNull$. Moreover, by (iv) and (v), we have $|V(R) \cap \bigcup_{a \in \N} \{\hat{r}_{a}^{i}\}| = \alephNull$ for each $i \in \{s+1, \dots, s+t\}$.
    Since for pairwise different $r, r' \in V(R) \cap \bigcup_{a \in \N} \{\hat{r}_{a}^{i}\}$, we have $h(r) \neq h(r')$, by (v), $K \cap G_i$ consists of arbitrarily large pairwise disjoint stars. Since $\{S_{j_i}\}_{i \in [s]} \cup \{\hat{S}_j^i\}_{i \in [t], j \in \N}$ were chosen to be pairwise vertex disjoint, there exists $K' \subseteq K$, such that $G_i \cap K'$ is a copy of $\SalephNull$ for $i \in [s]$ and $G_i \cap K'$ is a copy of $G_{\text{star}}$ for $i \in \{s+1, \dots, s+t\}$ as desired by Maker in $\MBAscInfStars{C}{s}{t}$.
\end{proof}

\begin{theorem}
    \label{thm:eq:Breaker}
    Given $C, s, t \in \N_0$. If Breaker has a winning strategy in $\MBfin{C}{s}{t}$, he has a winning strategy in $\MBAscInfStars{C}{s}{t}$.
\end{theorem}
\begin{proof}
    Let
    \begin{align*}
        \{G_i\}_{i \in [s+t]}, \
        \{S_i\}_{i \in [Cs]}, \
        \{r_i\}_{i \in [Cs]}, \
        \{\hat{S}_j^i\}_{i \in [t], j \in \N}, \
        \{\hat{r}_j^i\}_{i \in [t], j \in \N}
    \end{align*}
    and $K_C^s$ be defined as in \Cref{thm:eq:Maker}.
    According to his winning strategy in $\MBfin{C}{s}{t}$, suppose that within his first $N$ moves, Breaker has claimed an edge set $A \subseteq E(K_C^s)$ in the Maker-Breaker game on $E(K_C^s)$, such that for $\overline{A} \coloneqq E(K_C^s) \setminus A$, there exists a coloring $c \colon E(K_C^{(s)}(\overline{A})) \to [Cs+t]$, such that $r_i \in e$ for all $i \in [Cs]$ and $e \in c^{-1}(i)$ and Breaker has a winning strategy in $\MBauxM{c^{-1}(j_1)}{1}$ for all $j_1 \in [Cs]$ and in $\MBauxM{c^{-1}(j_2)}{2}$ for all $j_2 \in \{Cs+1, \dots, Cs+t\}$.

    For convenience, let $H_a$ be the $s$-uniform hypergraph induced by $c^{-1}(a)$ for each $a \in [Cs+t]$. Let $X$ be the set of vertices, which were not fresh after Breaker's $N^\text{th}$ move and observe that $|X| \leq 2N$.
    For all vertices $v$ in
    \begin{align*}
        B^* \coloneqq \bigcup_{i \in [s+t]} V(G_i) \setminus (V(K_C^s) \cup X),
    \end{align*}
    Breaker plays in parallel Maker-Breaker games depending on the corresponding vertex. Observe that for all $v \in B^*$, each of Breaker's winning strategies in $\MBauxM{H_a}{1}$ and $\MBauxM{H_a}{2}$ yields him a winning strategy in $\MBauxVertexM{H_a}{1}{v}$ and $\MBauxVertexM{H_a}{2}{v}$, respectively.

    Whenever Maker claims an edge $vw \in E(B)$ with $v \in V(S_k) \cap V(B^*)$ and $w \in V(H_k)$ for $k \in [Cs]$, Breaker claims an edge according to his winning strategy in $\MBauxVertexM{H_k}{1}{v}$. Observe that by doing so, Maker can connect at most $2N$ leaves of $S_k$ completely to an edge of $H_k$.

    Whenever Maker claims an edge $vw \in E(B)$ with $v \in V(\hat{S}_a^b)$ and $w \in V(H_{Cs+b}) \cup \{\hat{r}_a^b\}$ for $a \in \N$ and $b \in [t]$, Breaker claims an edge according to his winning strategy in the game $\MBauxVertexM{H_{Cs+b}}{2}{\hat{r}_a^b}$\footnote{in which Maker is allowed to pass (see \Cref{rem:passingAllowed})} as long as not all edges of $E(\{\hat{r}_a^b\}, H_{Cs+b})$ have been claimed.
    Once all edges of $E(\{\hat{r}_a^b\}, H_{Cs+b})$ have been claimed, let $M_a^b$ be the set of vertices $w \in V(H_{Cs+b})$ for which Maker has claimed the edge $\hat{r}_a^b w$.

    \begin{figure}[htbp]
        \centering
        \includegraphics[width=0.55\textwidth]{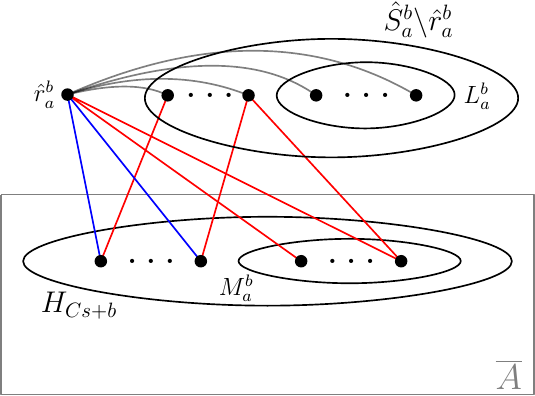}
        \caption{
            The red and blue edges have been claimed by Maker and Breaker, respectively. Whenever Maker tries to connect a finite star $\hat{S}^b_a$ to $H_{Cs+b}$, in a first stage, Maker can connect $\hat{r}_a^b$ only to a small subset $M_a^b \subseteq V(H_{Cs+b})$. Let $L^b_a \subseteq V(\hat{S}^b_a) \setminus \hat{r}^b_a$ be the set of leaves that is still fresh afterwards. If in a second stage for any $\ell \in L^b_a$, Maker manages to claim the edge $\hat{r}_a^b \ell$ before Breaker, he can prevent her from connecting $\ell$ to an edge of $H_{Cs+b}[M^b_a]$.
        }
        \label{pic:Breakereq}
    \end{figure}

    Let $L_a^b$ be the set of leaves of $\hat{S}_a^b$ for which no edge of $E(L_a^b, H_{Cs+b} \cup \{\hat{r}_a^b\})$ has been claimed after all edges of $E(\{\hat{r}_a^b\}, H_{Cs+b})$ have been chosen. In addition, observe that $|L_a^b| \geq (a - 1) - (|V(H_{Cs+b})| + N)$ and that for each $\ell \in L_a^b$, Breaker has a winning strategy in $\MBauxVertexB{H_{Cs+b}[M_a^b]}{1}{\ell}$ by definition of $\MBauxM{H_{Cs+b}}{2}$.
    If Maker does not claim $\hat{r}_a^b \ell$ as her first edge for a fresh $\ell \in L_a^b$, Breaker claims it. Suppose now that for each fresh $\ell \in L_a^b$, Maker claims $\hat{r}_a^b \ell$ at first. After Maker has claimed $\hat{r}_a^b \ell$, Breaker immediately claims an edge according to his winning strategy in $\MBauxVertexB{H_{Cs+b}[M_a^b]}{1}{\ell}$ and plays according to that strategy, whenever Maker claims an edge $\ell w$ with $w \in M_a^b$ in the future. Observe that for each $\ell \in L_a^b$, if Maker claims $\hat{r}_a^b \ell$, then she cannot connect $\ell$ completely to an edge of $H_{Cs+b}[M_a^b]$.

    Assume for a contradiction, that Maker has claimed a copy $K \subseteq B$ of $\KalephNull$, that wins her the game $\MBAscInfStars{C}{s}{t}$. By definition of Maker's goal in $\MBAscInfStars{C}{s}{t}$, we must have $v_1, \dots, v_s \in K$, such that $d_{G_i \cap K}(v_i) = \alephNull$ for each $i \in [s]$. Therefore, we obtain $e \coloneqq \{v_1, \dots, v_s\} \in K_C^{(s)}(\overline{A})$ and we set $k \coloneqq c(e)$.

    If $k \in [Cs]$, we have without loss of generality $v_1 = r_k \in e$. Due to Breaker's winning strategy in $\MBauxM{H_k}{1}$, Maker has not been able to connect more than $2N$ leaves of $S_k$ completely to $e$, which is an edge of $H_k$. Hence, $d_{G_i \cap K}(v_1) = d_{G_i \cap K}(r_k) \leq 2N$.

    If $k \in \{Cs+1, \dots, Cs+t\}$, without loss of generality, let $k = Cs+1$ and let $S \coloneqq \hat{S}^1_a \cap K$ for some $a \in \N$, such that $\hat{r}^1_a \in S$. By construction for all $\ell \in L^1_a$, Breaker has either claimed $\hat{r}^{1}_a \ell$, which results in $\ell \notin V(K)$ since $\hat{r}^{1}_a \in V(K)$, or Breaker has been able to prevent Maker from connecting $\ell$ completely to an edge $e$, which is an edge of $H_{Cs+1}[M_a^1]$. Since $e \subseteq V(K)$, at most $|V(H_{Cs+1})| + N$ leaves of $\hat{S}^1_a$ are contained in $K$. Hence, $\Delta(K \cap G_{s+1}) \leq |V(H_{Cs + 1})| +N$, which completes the proof.
\end{proof}

In order to give a sufficient condition for Maker to win $\MBfin{C}{s}{t}$, we need to apply the following result by Beck \cite[Theorem 2.4]{BeckRamseyGames2002} to our auxiliary game.
\begin{theorem}
    \label{thm:Beck02}
    For some $s \in \N$, let $\mathcal{F}$ be an $s$-uniform hypergraph. Maker has a winning strategy in $\MBauxM{\mathcal{F}}{1}$, if $|E(\mathcal{F})| > 2^{s-3} \cdot \Delta_2(\mathcal{F}) \cdot |V(\mathcal{F})|$, where
    \begin{align*}
        \Delta_2(\mathcal{F}) \coloneqq \max_{x,y \in V(\mathcal{F})}|\{e \in E(\mathcal{F}) \ | \ x,y \in e\}|.
    \end{align*}
\end{theorem}

Recall \Cref{def:KCs}. One can prove a similar result in the game $\MBauxM{\mathcal{F}}{2}$ for $\mathcal{F} \subseteq K_C^{(s)}$ being a sufficiently large subgraph.

\begin{lemma}
    \label{lem:Aux2Game}
    Let $\mathcal{F} \subseteq K_C^{(s)}$ be a $s$-uniform hypergraph, such that $|E(\mathcal{F})| \geq \frac{C^s}{2t}$ for some $s, t \in \N$ and $C \geq 2^{3s+1} \cdot t$. Then, Maker has a winning strategy in $\MBauxM{\mathcal{F}}{2}$.
\end{lemma}
\begin{proof}
    We begin by showing that there exist four pairwise vertex disjoint copies of $T_4^{(s)}$ in $\mathcal{F}$.
    Fix an enumeration of the partition classes $P_1, \dots, P_s$ of $K_C^{(s)}$ and suppose that for a fixed $0 \leq a \leq s-1$, every $\ell \leq a$ and all $(j_1, \dots, j_\ell) \in [4]^\ell$, we have already chosen pairwise different vertices $x_{(j_1, \dots, j_\ell)} \in P_\ell$, such that for
    \begin{align*}
        E_{(j_1, \dots, j_a)} \coloneqq \{e \in E(\mathcal{F}) \ | \ x_{(j_1)}, x_{(j_1, j_2)}, \dots,x_{(j_1, \dots, j_a)} \in e \},
    \end{align*}
    we have $|E_{(j_1, \dots, j_a)}| \geq \frac{C^{s-a}}{2^{a+1}t}$. This holds by construction for $a = 0$.
    Now, we fix a sequence $(i_1, \dots, i_a) \in [4]^a$. For $x \in P_{a+1}$, define
    \begin{align*}
        E_x \coloneqq \{e \in E(\mathcal{F}) \ | \ x_{(i_1)}, x_{(i_1, i_2)}, \dots,x_{(i_1, \dots, i_a)}, x \in e\}.
    \end{align*}
    Assume for a contradiction that there exist less than $4^s$ vertices in $P_{a+1}$, such that $|E_x| \geq \frac{C^{s-(a+1)}}{2^{a+2}t}$.
    By using that $C \geq 4^s \cdot 2^{a+2} \cdot t$ for all $a \leq s-1$ in $(*)$, we obtain that
    \begin{align*}
        |E_{(j_1, \dots, j_a)}|
         & = \sum_{x \in P_{a+1}} |E_x|
        \leq (C-4^s +1) \cdot \left(\frac{C^{s-(a+1)}}{2^{a+2}t} -1 \right) + (4^s-1) \cdot C^{s-(a+1)} \\
         & < \frac{C^{s-a}}{2^{a+2}t} + 4^s \cdot C^{s-(a+1)}
        = C^{s-(a+1)} \left( \frac{C}{2^{a+2}t} +4^s \right)                                            \\
         & \overset{(*)}{\leq} C^{s-(a+1)} \left( \frac{C}{2^{a+1}t} \right)
        \leq |E_{(j_1, \dots, j_a)}|.
    \end{align*}
    Hence, for all $(j_1, \dots, j_{a+1}) \in [4]^{a+1}$, we can choose $x_{(j_1, \dots, j_{a+1})} \in P_{a+1}$ pairwise disjoint, such that $|E_{(j_1, \dots, j_{a+1})}| \geq \frac{C^{s-(a+1)}}{2^{a+2}t}$.

    By definition of $C$, we have that for each $(j_1, \dots, j_s) \in [4]^s$, we have $E_{(j_1, \dots, j_s)} \neq \emptyset$ and that
    \begin{align*}
        \bigcup_{(j_1, \dots, j_s) \in [4]^s} E_{(j_1, \dots, j_s)} \subseteq E(\mathcal{F})
    \end{align*}
    corresponds to the edge set of four copies $T_1, \dots, T_4$ of $T_4^{(s)}$ as desired.

    Let $X = V(T_1) \cup \ldots \cup V(T_4)$.
    By \Cref{rem:connVertexToHalfTrees} and the equivalence of $\MBauxM{\mathcal{F}[X]}{2}$ and $\MBauxVertexM{\mathcal{F}[X]}{2}{x}$ for any $x \in B$, Maker can claim a vertex set $Y \subseteq X$ containing two copies of $T_2^{(s)}$ in the game $\MBauxM{\mathcal{F}[X]}{2}$. By \Cref{lem:connVertexToHalfTrees}, Maker has a winning strategy in $\MBauxB{\mathcal{F}[Y]}{1}$. This completes the proof.
\end{proof}

By combining the tools we have developed so far, we can finally give a sufficient condition on $C$, such that Maker has a winning strategy in $\MBfin{C}{s}{t}$.

\begin{theorem}
    \label{thm:MakerWinFinGame}
    Let $s \in \N$ and $t \in \N_0$. For $C = \max\{\lceil s^2 \cdot 2^{s-3} \rceil, 2^{3s+1} \cdot t\}$ and any $C' \geq 3C (s+1)^{Cs + 1}$, Maker has a winning strategy in $\MBfin{C'}{s}{t}$.
\end{theorem}
\begin{proof}
    Let $K_{C'}^s$ be as in \Cref{def:KCs} with enumerated vertices $V(K_{C'}^s) = \{r_1, \dots, r_{C's}\}$. We define an $s$-coloring of $V(K_{C'}^s)$ by coloring vertices from different partition classes with different colors. Since there exist at least $3C (s+1)^{Cs + 1}$ vertices in every partition class of $K_{C'}^s$, by \Cref{thm:lucagame}, Maker has a strategy to claim all edges of a $K_C^s \subseteq K_{C'}^s$ in at most $C (s+1)^{Cs + 1}$ moves.

    Suppose we are given $c \colon E(K_C^{(s)}) \to [C's+t]$, such that $r_j \in e$ for all $e \in c^{-1}(j)$ and $j \in [C's]$. Since $|V(K_C^{(s)})| = Cs$, assume that $c^{-1}(j) = \emptyset$ for $j \in \{Cs+1, \dots, C's\}$ up to relabeling vertices by the additional property of $c$. For simplicity, we now consider $c \colon E(K_C^{(s)}) \to [Cs+t]$, instead. For all $k \in [Cs+t]$, let $\mathcal{F}_k \coloneqq c^{-1}(k)$. By the pigeonhole principle, there exists $k' \in [Cs+t]$, such that $|E(\mathcal{F}_{k'})| \geq \frac{C^s}{2t}$, if $k' \in \{Cs+1, \dots, Cs+t\}$ or $|E(\mathcal{F}_{k'})| \geq \frac{C^{s-1}}{2s}$, if $k' \in [Cs]$.

    In the first case, Maker has a winning strategy in $\MBauxM{\mathcal{F}_{k'}}{2}$ by \Cref{lem:Aux2Game}, since $\mathcal{F}_{k'} \subseteq K_C^{(s)}$ with $|E(\mathcal{F}_{k'})| \geq \frac{C^s}{2t}$.
    In the latter case, we must have $r_{k'} \in e$ for all $e \in E(\mathcal{F}_{k'})$ by the additional property of $c$. In the game $\MBauxM{\mathcal{F}_{k'}}{1}$, Maker begins by claiming $r_{k'}$.
    Suppose that Breaker has claimed a vertex $x \in V(\mathcal{F}_{k'})$ in his first move. Let $\mathcal{F}$ be the $(s-1)$-uniform hypergraph, defined by $E(\mathcal{F}) \coloneqq \{e \setminus \{r_{k'}\} \ | \ e \in E(\mathcal{F}_{k'}), \ x \notin e\}$ and $V(\mathcal{F}) \coloneqq \bigcup_{e \in E(\mathcal{F})} e$. Observe that $|E(\mathcal{F})| \geq \frac{C^{s-1}}{2s} - C^{s-2}$. By using that $C \geq s^2 \cdot 2^{s-3}$, we can verify that \Cref{thm:Beck02} is applicable:
    \begin{align*}
        2^{(s-1)-3} \cdot \Delta_2(\mathcal{F}) \cdot |V(\mathcal{F})| \leq  2^{s-4} \cdot C^{s-3} \cdot C(s-1) < \frac{C^{s-1}}{2s} - C^{s-2} \leq |E(\mathcal{F})|.
    \end{align*}
    Hence, Maker has a winning strategy in $\MBauxM{\mathcal{F}}{1}$. By construction of $\mathcal{F}$, she has a winning strategy in $\MBauxM{\mathcal{F}_{k'}}{1}$ as desired by $\MBfin{C'}{s}{t}$.
\end{proof}

\begin{remark}
    For $s = 0$ and $C,t \in \N_0$, Maker has a winning strategy in $\MBfin{C}{s}{t}$ by definition of the game as follows. Since $K_{C}^{(s)}$ is defined as the complete $s$-partite graph, $E(K_{C}^{(0)})$ contains the empty edge $e'$. So for any coloring $c \colon E(K_{C}^{(0)}) \to [t]$, either Maker wins immediately, if $t = 0$ or there exists $i \in [t]$, such that $c^{-1}(i) = \{e'\}$. But since Maker does not have to claim any vertex in order to claim $e'$, she wins both $\MBauxM{c^{-1}(i)}{1}$ and $\MBauxM{c^{-1}(i)}{2}$ and therefore $\MBfin{C}{s}{t}$ for $s = 0$ and all $C,t \in \N_0$.
\end{remark}

With the previous remark and \Cref{thm:eq:Maker}, one obtains that Maker has a winning strategy in $\MBAscInfStars{C}{0}{t}$ for all $C, t \in \N_0$. Another way of observing this without having to check the proof of \Cref{thm:eq:Maker} with $s=0$ is the following.
\begin{remark}
    By \Cref{thm:eq:Maker} (with $s \geq 1$) and \Cref{thm:MakerWinFinGame}, one obtains, that for all $s, t \in \N$, there exists a $C_{s,t} \in \N$, such that Maker has a winning strategy in $\MBAscInfStars{C}{s}{t}$ for all $C \geq C_{s,t}$. Together with the observation that the game is monotone, i.e. that for all $s' \leq s$ and $t' \leq t$ every winning strategy for Maker in $\MBAscInfStars{C}{s}{t}$ is a winning strategy in $\MBAscInfStars{C}{s'}{t'}$, one obtains that Maker has a winning strategy in $\MBAscInfStars{C}{0}{t}$ for all $C, t \in \N_0$.
\end{remark}

\section{Proof of \Cref{thm:kColoring}}

We begin by proving the first part of \Cref{thm:kColoring}.
\begin{lemma}
    \label{lem:kColMaker}
    Given $s,t \in \N_0$ and $G_1, \dots, G_{s+t} \subseteq B$, such that $G_1, \dots, G_s$ are of order $C_1, \dots, C_s$ and $G_{s+1}, \dots, G_{s+t}$ are ascending, respectively. If
    \begin{align*}
        C_1, \dots, C_s \geq \max\{ 2^{t \cdot 2^{(1+o_s(1))3s}}, 2^{2^{(1+o_s(1))s}}\}
    \end{align*}
    then Maker has a winning strategy in $\MBcolorGraphs{s+t}$.
\end{lemma}
\begin{proof}
    Let $G_1, \dots, G_{s+t} \subseteq B$ be as above and $C' \coloneqq \min\{\frac{C_1}{s}, \dots, \frac{C_s}{s}\}$. For each $i \in [s]$, we choose pairwise disjoint sets $P_i \subseteq V(G_i)$, such that $|P_i| = C'$ and $d_{G_i}(v) = \alephNull$ for all $v \in P_i$.

    We can now recursively construct pairwise disjoint subgraphs $G'_k \subseteq G_k$, such that $G'_k$ is a copy of $G_{\text{inf}}^C$ with $P_k \subseteq G_k'$ for each $k \in [s]$ and $G'_k$ is a copy of $G_{\text{star}}$ for every $k \in \{s+1, \dots, s+t\}$. Applying \Cref{thm:(part)pattStar} with $G'_1, \dots, G'_{s+t}$ yields Maker a winning strategy in $\MBAscInfStars{C}{s}{t}$, which is also a winning strategy in $\MBcolorGraphs{s+t}$ as desired.
\end{proof}

In order to prove the second part of \Cref{thm:kColoring}, we need to introduce the following lemma about an auxiliary game (see \cite[Theorem 2.4.1]{PosGamesHef} for a proof).

\begin{definition}
    For $n,q \in \N$, let $\text{MB}(n,q)$ be the Maker-Breaker game in which Maker and Breaker alternately claim edges of $K_n$. Maker begins and wins if she can claim $K_q \subseteq K_n$. Breaker wins otherwise.
\end{definition}

\begin{lemma}
    \label{lem:auxPosGame}
    Given $q \in \N$ and the Euler number $e$. For every $n \leq \frac{q}{e} \cdot 2^{\frac{q}{2}-1}$, Breaker has a winning strategy in $\text{MB}(n,q)$.
\end{lemma}

We need the following result, which can be obtained by combining \cite[Proposition 5.2.2]{DiestelBible} and \cite[Theorem 8.1.3]{DiestelBible}.
\begin{lemma}
    \label{lem:chromaticNumber}
    Every (infinite) graph $H$ with bounded degree satisfies
    \begin{align*}
        \chi(H) \leq \Delta(H)+1.
    \end{align*}
\end{lemma}
Recall that $m(A)$ was defined just before \Cref{thm:kColoring} for a family of infinite graphs $G_1, \dots, G_k \subseteq B$ and $A \subseteq [k]$ as the size of a minimal vertex set $\{v_1, \dots, v_{m(A)}\}$, such that for each $j \in A$, there exists $v \in \{v_1, \dots, v_{m(A)}\}$ with $d_{G_j}(v) = \alephNull$.
\begin{lemma}
    Given $s,t \in \N_0$ and $G_1, \dots, G_{s+t} \subseteq B$, such that $G_1, \dots, G_s$ are of order $C_1, \dots, C_s$ and $G_{s+1}, \dots, G_{s+t}$ are ascending, respectively.
    For the Euler number $e$, if there exists a non-empty set $A \subseteq [s]$, such that
    \begin{align*}
        \sum_{j \in A} C_j \leq  \frac{m(A)}{e} \cdot 2^{\frac{m(A)}{2}-1},
    \end{align*}
    then Breaker has a winning strategy in $\MBcolorGraphs{s+t}$. Moreover, Breaker's strategy ensures that $G' = \bigcup_{j \in A} G_j \cap K'$ is a graph of order less than $m(A)$ for each infinite clique $K' \subseteq B$ for which Maker has claimed $E(K')$.
\end{lemma}
\begin{proof}
    Suppose, we are given $G_1, \dots, G_{s+t}$ as above and let $A \subseteq [s]$ be a non-empty set, such that $\sum_{j \in A} C_j \leq \frac{m(A)}{e} \cdot 2^{\frac{m(A)}{2}-1}$. Observe that for \Cref{lem:auxPosGame}, it is sufficient to show that Breaker has a winning strategy in $\text{MB}_{\text{col}}(\{G_i\}_{i \in A})$ and that $G' = \bigcup_{j \in A} G_j \cap K'$ is a graph of order less than $m(A)$ for each infinite clique $K'$ claimed by Maker. Hence, we assume without loss of generality that $A = [s]$ and $t = 0$.

    We define $X \coloneqq \{v \in V(B) \ | \ \exists i \in [s] \colon d_{G_i}(v) = \alephNull \}$ and $K_X \subseteq B$ as the complete graph induced by $X$. Since $|X| \leq \sum_{j \in A} C_j \leq \frac{m(A)}{e} \cdot 2^{\frac{m(A)}{2}-1}$, Breaker has a winning strategy in $\text{MB}(|X|, m(A))$ by \Cref{lem:auxPosGame}. Therefore, he plays according to it on $K_X$. Suppose that after $N$ moves, Breaker has claimed an edge set $E_X \subseteq E(K_X)$, such that $E(K_X) \setminus E_X$ does not contain a clique of size $m(A)$.

    Let $G \coloneqq (G_1 \cup \ldots \cup G_s) \setminus K_X$ and observe that $G$ is bounded by construction. We are going to define a pairing strategy for Breaker on a subset of $E(B \setminus K_X)$. First, recall \Cref{def:linegraph,def:distGraph}. Observe that $\Delta(L(G)^2) \leq (2 \Delta(G))^2 < \alephNull$ and set $k = \Delta(L(G)^2) +1$. By \Cref{lem:chromaticNumber}, there exists a proper vertex coloring $c' \colon V(L(G)^2) \to [k]$ with $c'(u) \neq c'(v)$ for each $\{u,v\} \in E(L(G)^2)$. By definition of $L(G)^2$, this corresponds to a proper edge coloring $c \colon E(G) \to [k]$ with the properties that
    \begin{enumerate}
        \item for $e, e' \in E(G)$ with $e \cap e' \neq \emptyset$, we have $c(e) \neq c(e')$,
        \item for all $i,j \in [k]$ and each $e \in c^{-1}(i)$, there exists at most one edge $e' \in c^{-1}(j)$, such that $e \cap e' \neq \emptyset$.
    \end{enumerate}
    Fix an orientation of the edge set $E(G)$. Formally, we define functions $s', t' \colon E(G) \to V(G)$, such that for each $e \in E(G)$ we have $e = s'(e) t'(e)$. If $e = uv$ with $u = s'(e)$ and $v = t'(e)$, then we indicate this by saying that $e$ is \emph{oriented as} $\vv{uv}$.

    For all $i,j \in [k]$, we partition $c^{-1}(i)$ into the sets $M_1(i,j), \dots, M_4(i,j)$ as follows. For an edge $e \in c^{-1}(i)$, we let $e \in M_1(i,j)$, if $e \cap e' = \emptyset$ for all $e' \in c^{-1}(j)$. Otherwise by (ii), there exists a unique edge $e^* \in c^{-1}(j)$ with $e \cap e^* \neq \emptyset$. Then, we set
    \begin{itemize}
        \item $e \in M_1(i,j)$, if $e \cap e^* = s'(e) = s'(e^*)$,
        \item $e \in M_2(i,j)$, if $e \cap e^* = s'(e) = t'(e^*)$,
        \item $e \in M_3(i,j)$, if $e \cap e^* = t'(e) = s'(e^*)$,
        \item $e \in M_4(i,j)$, if $e \cap e^* = t'(e) = t'(e^*)$.
    \end{itemize}
    This defines a partition $c^{-1}(i) = M_1(i,j) \sqcup M_2(i,j) \sqcup M_3(i,j) \sqcup M_4(i,j)$.
    Using this, we define a coloring $\tilde{c} \colon E(G) \to [k] \times [4]^k$ by setting
    \begin{align*}
        \tilde{c}(e) = (c(e), i_1, \dots, i_k), \text{ if } e \in M_{i_1}(c(e),1) \cap \ldots \cap M_{i_k}(c(e),k).
    \end{align*}
    Finally, we define a pairing of edges
    \begin{align*}
        \mathcal{P} = \big\{ \{ux, vw\} \in E(B)^{(2)} \ | \ \tilde{c}(uv) = \tilde{c}(wx),~s'(uv) = u \text{~and~} s'(wx) = w \big\}.
    \end{align*}
    \begin{figure}[htbp]
        \centering
        \includegraphics[width=0.20\textwidth]{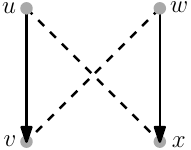}
        \caption{For $uv$ and $wx$ with $\tilde{c}(uv) = \tilde{c}(wx)$ directed as above, we add the dashed crossing edges to $\mathcal{P}$.}
        \label{pic:PairingDef}
    \end{figure}
    Since $uv$ and $wx$ cannot be incident by (i), $\mathcal{P}$ is well-defined.

    We have to verify that each edge of $B$ is contained in at most one pair of $\mathcal{P}$. So, assume for contradiction that we have $\{vw, ux\}, \{vw, yz\} \in \mathcal{P}$ for $u,v,w,x,y,z \in V(B)$ with $ux \neq yz$. We may assume without loss of generality that $u \notin \{y,z\}$.

    If $x \in \{y,z\}$, let without loss of generality $x = y$. By symmetry of $v$ and $w$, we can assume that $\tilde{c}(uv) = \tilde{c}(wx)$. If $\tilde{c}(wy) = \tilde{c}(vz)$, we have in particular $c(uv) = c(wx) = c(vz)$ contradicting (i) and if $\tilde{c}(vy) = \tilde{c}(wz)$, we have $c(vx) = c(wz)$, which contradicts (ii) for $wx \in c^{-1}(j)$ with $j \in [k]$. Hence, we must have $x \notin \{y,z\}$.

    \begin{figure}[H]
        \centering
        \begin{subfigure}{0.43\textwidth}
            \centering
            \includegraphics[width=0.6\linewidth]{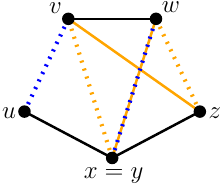}
            \caption{The black edges are the ones in $\bigcup \mathcal{P}$. The blue dotted edges are the ones that cause $\{vw, ux\} \in \mathcal{P}$. Either the orange solid edges or the orange dotted edges cause $\{vw, yz\} \in \mathcal{P}$.}
        \end{subfigure}
        \hspace{0.01\textwidth}
        \begin{subfigure}{0.52\textwidth}
            \centering
            \includegraphics[width=0.7\linewidth]{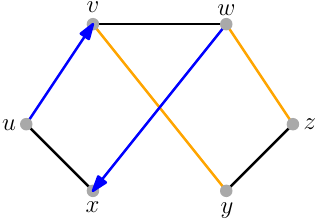}
            \caption{The black edges are from $\bigcup \mathcal{P}$. The blue edges and the orange edges are the ones that cause $\{vw, ux\} \in \mathcal{P}$ and $\{vw, yz\} \in \mathcal{P}$, respectively up to relabeling vertices. The blue edges must be directed similarly to \Cref{pic:PairingDef}.}
        \end{subfigure}
        \caption{The case analysis based on whether $x \in \{y,z\}$  or $x \notin \{y,z\}$.}
    \end{figure}

    By symmetry of $u$ and $x$ (respectively $y$ and $z$), we can assume that $\tilde{c}(uv) = \tilde{c}(wx)$ (respectively $\tilde{c}(vy) = \tilde{c}(wz)$). Up to relabeling colors, assume that $c(uv) = c(wx) = 1$ and $c(vy) = c(wz) = 2$. By symmetry of $v$ and $w$ and by definition of $\tilde{c}$, we can assume that $uv$ is oriented $\vv{uv}$ and $wx$ is oriented $\vv{wx}$. But then $uv \in M_3(1,2) \cup M_4(1,2)$, while $wx \in M_1(1,2) \cup M_2(1,2)$ contradicting $\tilde{c}(uv) = \tilde{c}(wx)$.

    From his $(N+1)^\text{st}$ move onward Breaker plays according to the following pairing strategy. Whenever Maker claims an edge $e$ with $\{e,e'\} \in \mathcal{P}$, Breaker claims $e'$, if possible. Otherwise, Breaker claims any element of $E(B)$. Assume that Maker has claimed all edges of an infinite clique $K \subseteq B$ after countably many moves. Due to his play according to the winning strategy in $\text{MB}(|X|, m(A))$ on $K_X$, we have that $G' \coloneqq \bigcup_{j \in A} G_j \cap K$ is of order less than $m(A)$ as desired.

    Assume now for contradiction that $K$ is winning for Maker in $\MBcolorGraphs{s+t}$. Since $G'$ is of order less than $m(A)$, there must exist $j \in A$, such that $G_j \cap K$ is infinite, but bounded. Hence, $E(G_j \cap (K \setminus K_X))$ is infinite and since $G_j \cap (K \setminus K_X) \subseteq G$, also $E(G \cap K)$ must be infinite. But according to Breaker's pairing strategy, for each $a \in [k] \times [4]^k$, Maker's clique $K$ contains at most $N$ pairs $\{e, e'\} \in \mathcal{P}$ with $\tilde{c}(e) = \tilde{c}(e') = a$. Therefore, $K \cap G$ contains at most $N$ edges of $\tilde{c}^{-1}(a)$ for every $a \in [k] \times [4]^k$, contradicting that $E(G \cap K)$ was infinite.
\end{proof}

\section{Additional results}
In this section, we prove \Cref{thm:2ColoringCol}, \Cref{cor:patternCliqueStar} and \Cref{thm:partpattGame}.
For convenience, we rephrase \Cref{thm:2ColoringCol} in terms of graphs. Observe that \Cref{thm:2ColoringCol} and \Cref{thm:2ColoringGraph} are equivalent.
\begin{theorem}
    \label{thm:2ColoringGraph}
    Given an infinite graph $G \subseteq B$. If either $G$ or its complement $\overline{G}$ is bounded, Breaker has a winning strategy in $\text{MB}_{\text{col}}(G, \overline{G})$. Otherwise, Maker has a winning strategy in $\text{MB}_{\text{col}}(G, \overline{G})$.
\end{theorem}
\begin{proof}
    For convenience, set $G_1 = G$ and $G_2 = \overline{G}$. Observe that by Ramsey's theorem \cite{RamseysThm}, we can assume without loss of generality that $G_2$ contains a copy $K'$ of $\KalephNull$, in particular $G_2$ is ascending. If $G_1$ is bounded, one obtains that Breaker has a winning strategy in $\text{MB}_{\text{col}}(G_1, G_2)$ by applying \Cref{thm:kColoring} with $s = t = 1$ and $A = \{1\}$.

    If $G_1$ is ascending, we obtain that Maker has a winning strategy in $\text{MB}_{\text{col}}(G_1, G_2)$ by applying \Cref{thm:kColoring} with $s=0$ and $t=2$. If $G_1$ is of order $C$ for some $C \in \N$, we can assume there exists a vertex $r \in V(B)$ with $d_{G_1}(r) = \alephNull$.

    If $d_{G_2}(r) = \alephNull$, define $c_{\text{vert}} \colon V(B) \to [2]$ by setting $c_{\text{vert}}(r) = 1$ and $c_{\text{vert}}(x) = i$, if $rx \in G_i$ for $i \in [2]$. Applying \Cref{thm:lucagame} with $c_{\text{vert}}$ and $r$ yields a copy of $K$, such that $d_{(G_1 \cap K)}(r) = d_{(G_2 \cap K)}(r) = \alephNull$. If $N_{G_2}(r)$ is finite, set the board $B' = \{r\} \cup (N_{G_1}(r) \cap K')$, which is still infinite. Applying \Cref{thm:lucagame} with $r$ and $k=1$ this time, yields Maker a copy of $\KalephNull$ in $B'$ containing infinitely many edges of $G_1$ and $G_2$.
\end{proof}

\begin{proof}[Proof of \Cref{cor:patternCliqueStar}]
    Let $G_1, \dots, G_k \subseteq B$ be given as in \Cref{cor:patternCliqueStar}. For every $i \in [k]$ for which $G_i$ is a copy of $G_{\text{star}}$, set $G_i^* = G_i$. For each $G_i$ that is a copy of $G_{\text{cli}}$ and each clique $K_{\text{fin}} \subseteq G_i$, fix a maximal star $S_{\text{fin}} \subseteq K_{\text{fin}}$ and set $G_i^* = \bigcup S_{\text{fin}}$. Applying \Cref{thm:(part)pattStar} with $s =0$ and $t = k$ on $G_1^*, \dots, G_k^*$ will yield the desired infinite clique for Maker in \Cref{cor:patternCliqueStar}.
\end{proof}

\begin{proof}[Proof of \Cref{thm:partpattGame}]
    Note that applying \Cref{thm:kColoring} with $s \in \N$, $t = 0$ and $G_1, \dots, G_s$ each being the vertex disjoint union of $\lfloor \frac{1}{e} \cdot 2^{\frac{s}{2}-1} \rfloor$ copies of $\SalephNull$, we obtain with $A = [s]$, that Breaker has a winning strategy in $\MBcolorGraphs{s}$. By the additional consequence of \Cref{thm:kColoring}, $\bigcup_{j \in A} G_j \cap K$ is a graph of order less than $s$ for each copy $K$ of $\KalephNull$ that Maker claims.
    This immediately implies
    \begin{align}
        \label{eq:lowerBound}
        \MBpartpatternNumber{1}{\SalephNull}{s} > \frac{1}{e} \cdot 2^{\frac{s}{2}-1}.
    \end{align}
    Moreover, applying \Cref{thm:(part)pattStar} with $t=0$ yields
    \begin{align}
        \label{eq:upperBound}
        \MBpartpatternNumber{s}{\SalephNull}{1} \leq  2^{2^{(1+ o(1))s}}.
    \end{align}
    With the observations above, the proof is a one-liner:
    \begin{align*}
        \frac{1}{ke} \cdot 2^{\frac{k \ell}{2}-1}
        \overset{(\ref{eq:lowerBound})}{\leq}
        \frac{1}{k} \cdot \MBpartpatternNumber{1}{\SalephNull}{k \ell}
        \overset{\ref{lem:PartPatt:LowerAndUpper}}{\leq}
        \MBpartpatternNumber{k}{\SalephNull}{\ell}
        \overset{\ref{lem:PartPatt:LowerAndUpper}}{\leq}
        \ell \cdot \MBpartpatternNumber{k \ell}{\SalephNull}{1}
        \overset{(\ref{eq:upperBound})}{\leq}
        \ell \cdot 2^{2^{(1+ o(1))k \ell}}
    \end{align*}
\end{proof}

\section{Open problems}
We conclude with open problems of two different types arising from this work.
The first type is about the game $\MBcolor{k}{c}$. We believe that the following problem is a natural one to consider in order to improve the understanding of the game.
\begin{problem}
\label{prob:3color}
Provide a complete characterization of the game $\MBcolor{3}{c}$. Determine which player has a winning strategy in $\MBcolor{3}{c}$ for each coloring $c \colon E(\KalephNull) \to [3]$.
\end{problem}
With the methods provided in this paper, one can determine in many cases which player has a winning strategy in $\MBcolor{3}{c}$, but some key cases remain open. Let $G_1, G_2, G_3$ be the graphs induced by $c^{-1}(\{1\}), c^{-1}(\{2\}), c^{-1}(\{3\})$ in $\MBcolor{3}{c}$, respectively. Without loss of generality, assume that $G_3$ contains a copy of $\KalephNull$ by Ramsey's theorem. We believe that solving the following is a good intermediary step in order to solve \Cref{prob:3color}.
\begin{problem}
Let $G_1$ consist of one or two vertex disjoint copies of $\SalephNull$ and $G_2$ be ascending. Determine the winner in $\MBcolor{3}{c}$ based on $G_1$ and the structure of $G_2$.
\end{problem}
Consider the complete $r$-uniform infinite board $B^{(r)}$ for some $r \geq 3$. Another way of generalizing the game $\MBcolor{k}{c}$ is the following.
\begin{problem}
Given $r \geq 3$. For which colorings $c \colon E(B^{(r)}) \to [2]$, does Maker have a strategy to claim a complete $r$-uniform infinite graph $K \subseteq B^{(r)}$ with $|c^{-1}(i) \cap E(K)| = \alephNull$ for all $i \in [2]$.
\end{problem}

The second type of problems is about the pattern preserving game. For an infinite graph $H$, let $G_H$ be the disjoint union of $\alephNull$ many copies of $H$.
\begin{problem}
For which $k \in \N$ and graphs $H$ does Maker have a winning strategy in $\MBpatternGraphs{k}$ with $G_1, \dots, G_k$ being pairwise disjoint copies of $G_H$?
\end{problem}
Note that $H$ needs to be infinite here, since Breaker has a simple pairing strategy that is winning for him, if $H$ is finite and has at least two edges. For $H = \KalephNull$, it has been shown in \cite[Theorem 4.1]{Bowler_2023} that Maker has a winning strategy for every $k \in \N$. A reasonable choice to consider next is the case $H = \SalephNull$.
\begin{problem}
\label{prob:infStars}
For which $k \in \N$, does Maker have a winning strategy in $\MBpatternGraphs{k}$ for $G_1, \dots, G_k$ being pairwise disjoint copies of $G_{\SalephNull}$?
\end{problem}
As far as we know, \Cref{prob:infStars} is open for all $k \in \N$.

\textbf{AI disclosure:} The authors have used AI (Claude Opus 4.8) only for polishing the writing. All mathematical content is entirely the work of the authors.

\printbibliography

\end{document}